\documentclass[12pt]{amsart}

\usepackage{amsfonts, dsfont}
\usepackage{amssymb, mathrsfs}
\usepackage{mathtools}
\usepackage{amsmath}
\usepackage{enumerate}
\usepackage{mathtools}
\numberwithin{equation}{section}

\usepackage{graphicx, hyperref, xcolor, wrapstuff, longtable}
\usepackage{tikz}
\usetikzlibrary{decorations.markings}
 \usepackage[top=0.8in, left=0.75in, right=0.75in, bottom=0.5in]{geometry}
 \usepackage{microtype}
 \usepackage[normalem]{ulem}

\theoremstyle{plain}
\newtheorem{thm}{Theorem}[section]
\newtheorem{lem}[thm]{Lemma}
\newtheorem{cor}[thm]{Corollary}
\newtheorem{prop}[thm]{Proposition}

\theoremstyle{definition}
\newtheorem{defin}[thm]{Definition}
\newtheorem{example}[thm]{Example}

\theoremstyle{remark}
\newtheorem{rem}[thm]{Remark}

\newcommand{\Q}{\mathbb{Q}}
\newcommand{\R}{\mathbb{R}}
\newcommand{\N}{\mathbb{N}}
\newcommand{\Z}{\mathbb{Z}}
\newcommand{\norm}[1]{\lVert #1\rVert}
\newcommand{\TV}[1]{\norm{#1}_{\textrm{TV}}}
\newcommand{\Lp}[2][p]{\norm{#2}_{L^{#1}}}
\newcommand{\abs}[1]{\left\lvert #1\right\rvert}
\newcommand{\HOne}[1]{\mathcal{H}^1(#1)}
\newcommand{\inner}[2]{\langle #1, #2\rangle}
\newcommand{\wto}{\rightharpoonup}
\newcommand{\flatto}{\overset{\flat}{\to}}
\newcommand{\currentto}{\overset{\mathcal{D}^\ast}{\wto}}
\newcommand{\wstarto}{\overset{\ast}{\wto}}
\mathchardef\mhyphen="2D
\newcommand{\approxto}[1][p]{\overset{{#1}\mhyphen\textrm{approx.}}{\to}}
\newcommand{\Wpflatto}{\overset{W^{1, p}\mhyphen\flat}{\wto}}
\newcommand{\DELp}{E}

\newcommand{\ELp}{\mathcal{E}}
\newcommand{\Fcur}{\mathbb{F}}

\newcommand{\twocurrent}{\mathfrak{F}}
\newcommand{\Lqform}{\mathfrak{F}^\ast}
\newcommand{\Mcur}{\mathbb{M}}
\newcommand{\MLp}{M}
\newcommand{\overstar}[1]{\overset{\ast}{#1}\vphantom{#1}}

\DeclareMathOperator{\spt}{spt}

\DeclareMathOperator{\esssup}{ess\, sup}
\DeclareMathOperator{\In}{In}
\DeclareMathOperator{\Out}{Out}
\DeclareMathOperator{\Lip}{Lip}

\DeclareMathOperator{\dist}{d}
\DeclareMathOperator{\Path}{Path}

\newcommand{\WStarPath}[1][p]{\overstar{W}^{1,#1}\mhyphen\Path}
\DeclareMathOperator{\DPath}{\Path_D}

\newcommand{\WDPath}[1][p]{W^{1,#1}\mhyphen\DPath}

\newcommand{\meas}{\mathcal{M}}
\newcommand{\vecmeas}{\mathcal{M}^1}
\newcommand{\atomic}{\mathcal{P}^a}
\newcommand{\prob}{\mathcal{P}}
\newcommand{\domain}{\R^n}
\newcommand{\compact}{K}
\newcommand{\covecfield}{C_0(\domain; \Lambda^1(\domain))}

\begin{document}       
\title{Time-dependent branched transport} 
\author[]{Jun Kitagawa}
\address{Department of Mathematics, Michigan State University, East Lansing, MI 48824, USA}
\email{kitagawa@math.msu.edu}

\author[]{Cecilia Mikat}
\address{Department of Mathematics, Michigan State University, East Lansing, MI 48824, USA}
\email{mikatcec@msu.edu}

                \thanks{JK was supported in part by National Science Foundation grant DMS-2246606}
                
\begin{abstract}
We develop a new framework for branched transport between probability measures which are allowed to vary in time. This framework can be used to model problems where the underlying transportation network displays a branched structure, but the source and target mass distributions can change over time, such as road networks or circulatory systems. We introduce the notion of time-dependent transport paths along with associated energies and distances, and prove existence of transport paths whose energy achieves the distance. We also show the time-dependent transport yields a metric structure on subsets of appropriately defined measure-valued Sobolev spaces, which should be a result of interest to the wider mathematical community.
\end{abstract}

 \maketitle
\section{Introduction}
There are many known ways to metrize some kind of weak convergence on spaces of probability measures (see \cite[P.97]{Villani09book}), often in a variational way; metrics from the Monge--Kantorovich problem\footnote{Sometimes known as the \emph{Wasserstein distances}.} are one such family, based on the notion of transportation of mass. However, if one wishes to consider variationally defined metrics on spaces of \emph{measure valued functions}, a naive implementation will result in separately minimizing a family of stationary problems, which does not yield a meaningful geometric structure. In this paper, we adapt an approach from the \emph{branched transport problem} (\cite{QXia2003, MaddalenaSoliminiMorel03, BernotCasellesMorel05}) which turns out to yield a more meaningful structure in the setting of measure valued functions.

Branched transport is an extension of the discrete Gilbert--Steiner problem (see \cite{GilbertPollak68}), to the setting of general measures, which was given an Eulerian formulation in \cite{QXia2003} and a Lagrangian formulation in \cite{MaddalenaSoliminiMorel03} (later extended in \cite{BernotCasellesMorel05}). Since then, there have been numerous works on the problem, see \cite{BernotCasellesMorel09book} for an overview. It should also be noted the geometric structure arising from branched transport is a recent topic of great interest in the wider mathematical community (\cite{ColomboDeRosaMarchese21, ColomboDeRosaMarchesePegonProuff22, LohmannSchmitzerWirth22, CaldiniMarcheseSteinbruchel23}). The problem transports one probability measure to another, yielding solutions which exhibit ``branching'' where mass tends to be gathered together from source locations, transported for a period of time, then redistributed to the target locations once close. This is observed in many real world structures, such as circulation systems in animals and plants, and urban transportation networks; in particular branched transport yields more realistic transport structures than the Monge--Kantorovich formulation in these settings. Our goal is to develop a model of mass transport between distributions which may change with time, while retaining the physical characteristics of branched transport. 
Such phenomena are abundant in real life, for example in urban transportation various parts of the population move from one location to another, but these demands change with time of day; people may move from home to work in the mornings and return in the evenings, or go to the locations of other services at various times of the day. Another case is commercial airline travel, which naturally displays branched transport behavior (commercial flights usually transfer through hub airports) but flight traffic patterns can vary with time, due to change in customer demands, and also because of regulations related to airport usage. Finally, circulatory systems in organisms prioritize flow to different organs based on rest and activity levels, naturally following cyclical rhythms.

We begin with some definitions to state our main results, however some objects will require more in-depth background to introduce, hence the exact definitions are deferred to Sections~\ref{sec: discrete} and~\ref{sec: nondiscrete} below. The notation is unfortunately heavy, as the concepts we deal with are complicated in nature, but we have provided a notation table at the end of this section. We introduce metrics on what heuristically are ``weak-$\ast$ Bochner--Sobolev spaces'' (Definition~\ref{def: weak star sobolev}) taking values in compcatly supported probability measures, denoted $\overstar{W}^{1, p}(\prob_c)$ (see~\eqref{eqn: sobolev spaces}). The metrics are defined via minimization of an energy based on the following, inspired by work of Brancolini and Wirth:
\begin{defin}[\cite{BrancoliniWirth18}]\label{def: transportation cost}
    A \emph{transportation cost} is $\tau: [0, \infty)\to [0, \infty)$ satisfying:
    \begin{enumerate}[(a)]
        \item $\tau(0)=0$, $\tau>0$ on $(0, \infty)$,
        \item $\tau$ is nondecreasing,
        \item $\tau(s_1+s_2)\leq \tau(s_1)+\tau(s_2)$ for any $s_1$, $s_2\in [0, \infty)$,
        \item $\tau$ is lower semi-continuous.
    \end{enumerate}
    We say $\tau$ is \emph{admissible} if for some concave $\beta: [0, \infty)\to [0, \infty)$ with $\tau\leq \beta$ everywhere, 
    \begin{align*}
        \int_0^1s^{\frac{1}{n}-2}\beta(s)ds<\infty.
    \end{align*}
\end{defin}
Then for $1<p\leq \infty$, $\lambda\geq 0$, and an admissible transportation cost $\tau$, we define an energy $\ELp^{\tau, p}_\lambda$ (Definition~\ref{def: cts system cost function})   on $\WStarPath(\mu^+, \mu^-)$ (Definition~\ref{def: approximating seq}), the latter a space of maps taking values in vector measures that satisfy a notion of transport between two maps $\mu^\pm\in \overstar{W}^{1, p}(\prob_c)$; the energies are constructed as lower semi-continuous envelopes of discrete energies as in \cite{QXia2003, BrancoliniWirth18}. Using these energies we define our metric.
\begin{defin}\label{def: distance}
    Let $1< p\leq \infty$, $\lambda\geq 0$,  and $\tau$ be a transportation cost. Given any $\mu^\pm\in \overstar{W}^{1, p}(\prob_c)$, we define
    \begin{align*}
        \dist^{\tau, p}_{\lambda}(\mu^+, \mu^-):&=
        \inf_{T\in \WStarPath(\mu^+, \mu^-)}\ELp^{\tau, p}_{\lambda}(T).
    \end{align*}
\end{defin}
In our first main result, we show existence of minimizers for $\dist^{\tau, p}_\lambda$. Below, $\overstar{W}^{1, p}(\prob(\compact))$ is the set of maps in $\overstar{W}^{1, p}(\prob_c)$ with support in a fixed compact set $K$ and $[[E, V_t, \theta_t]]$ is notation for a rectifiable current (see~\eqref{eqn: sobolev spaces} and Definition~\ref{def: m-rectifiable sets}).
\begin{thm}\label{thm: minimizers exist}
Suppose $1<p\leq \infty$, $\lambda>0,$ and $\compact\subset \domain$ with $\compact$ compact and convex. 
\begin{enumerate}[(1)]
    \item\label{eqn: minimizer exists} If $\tau$ is an admissible transportation cost, for any $\mu^\pm\in \overstar{W}^{1, p}(\prob(\compact))$, there exists \\$T\in \WStarPath(\mu^+, \mu^-)$ satisfying
\begin{align*}
    \dist^{\tau, p}_{\lambda}(\mu^+, \mu^-)
    &=\ELp^{\tau, p}_\lambda(T),
\end{align*}
which is H\"older continuous from $[0, 1]$ to $(\mathcal{M}(\domain; \Lambda_1(\R^n)), \TV{\cdot})$. 
    \item\label{eqn: 1 rectifiable} If
     \begin{align}\label{eqn: tau at zero condition}
        \lim_{s\searrow 0}\frac{\tau(s)}{s}=\infty,
    \end{align}
there exists a countably $1$-rectifiable set $E\subset \compact$ such that for all $t\in [0, 1]$,
 \begin{align*}
     T[t]=[[E, V_t, \theta_t]],
 \end{align*}
 where $V_t$ is a unit length vector field orienting the approximate tangent space of $E$, \\and $\theta_t\in L^1_{loc}(\mathcal{H}^1\vert_E; \R)$.
\end{enumerate}
\end{thm}
From a modeling standpoint, the existence of a countably $1$-rectifiable $E$ independent of $t$ in Theorem~\ref{thm: minimizers exist}~\eqref{eqn: 1 rectifiable} is crucial: this corresponds to saying that for minimizers obtained above, the \emph{flow} of mass (people in a transportation network, nutrients in a circulation system) may vary with time, but the underlying \emph{physical network} itself (roads, blood/vascular vessels) does not change with time.

We also verify $\dist^{\tau, p}_\lambda$ is indeed a metric.
\begin{thm}\label{thm: metric}
    If $\tau$ is an admissible transportation cost, $\lambda>0$, and $1<p\leq\infty$, then $\dist^{\tau, p}_\lambda$ is a metric on $\overstar{W}^{1, p}(\prob_c)$.
\end{thm}
We actually go further, proving that convergence in $\dist^{\tau, p}_\lambda$ is equivalent to a notion of weak convergence which is essentially $W^{1, p}$ convergence of flat norms (Theorem~\ref{thm: weak convergence}), and if the measure valued functions $\mu^\pm$ are independent of time, the metric is equal to the \emph{cost distance} of \cite{BrancoliniWirth18} in the time-independent case (Theorem~\ref{thm: stationary problem}).

Our proofs involve a surprisingly subtle combination of vector valued integration with geometric measure theory, and the results are far more delicate than a simple extension of the branched transport theory in the stationary case.  We are mainly inspired by the Lagrangian approach introduced by Xia in \cite{QXia2003} and generalized by Brancolini and Wirth in \cite{BrancoliniWirth18}; the choices in definitions and their implications (both practical and mathematical) will be discussed below in Subsection~\ref{subsec: model choice} after the necessary background.

In Section~\ref{sec: background} we recall some of the central notions from vector-valued measures, currents, and Bochner integration theory that form the foundation of our setup. In Section~\ref{sec: discrete}, we consider a time-dependent analogue of the Gilbert--Steiner problem, i.e. transportation between atomic probability measure-valued functions. Then in Section~\ref{sec: nondiscrete} we extend the framework to more general probability measure-valued functions and provide a discussion motivating our various definitions, and prove Theorem~\ref{thm: minimizers exist}. In Section~\ref{sec: metric} we prove Theorem~\ref{thm: metric} and give a characterization of convergent sequences, and in Section~\ref{sec: consistency} we show consistency with the time-independent problem.

{\small
\begin{longtable}{p{2cm}p{8cm}p{1.75cm}}\hline
Notation&Description&Definition\\ \hline
$\meas$ & Radon measures &\quad \\
$\vecmeas$ & Vector valued Radon measures &\quad \\
$\overstar{W}^{1, p}(\prob_c)$ & Weak-$\ast$ Sobolev maps valued in probability measures with a fixed compact support & \eqref{eqn: sobolev spaces}\\
$\dist^{\tau, p}_{\lambda}$&Metric on $\overstar{W}^{1, p}(\prob_c)$&Definition~\ref{def: distance}\\
$\overstar{W}^{1, p}(\prob(\compact))$ & Maps in $\overstar{W}^{1, p}(\prob_c)$ whose support are contained in $\compact$ & \eqref{eqn: sobolev spaces}\\
$\Mcur$, $\Fcur$&Mass and flat norm of a current& Definition~\ref{def: mass and flat norm of currents}\\
$[[\tilde E, \tilde{V}, \tilde \theta]]$& The rectifiable current $\tilde\phi\mapsto \int_{\tilde E}\tilde\theta(x)\inner{\tilde{V}(x)}{\tilde \phi(x)}d\mathcal{H}^m(x)$&Definition~\ref{def: m-rectifiable sets}\\
$\twocurrent$&Converts vector measure valued maps to two currents&Definition~\ref{def: map to two forms}\\
$W^{1, p}(\atomic)$&Sobolev maps into atomic probability measures& \eqref{eqn: sobolev atomic}\\
$W^{1, p}\mhyphen\Path_D$&Discrete Sobolev transport paths&Definition~\ref{eqn: vector measure rep}\\
$\Mcur^\tau$&$\tau$-mass of rectifiable current&Definition~\ref{def: tau mass}\\
$\DELp^{\tau, p}_{\lambda}$&Energy on discrete transport paths&Definition~\ref{def: discrete energy}\\
$\Wpflatto$& $W^{1, p}$-flat convergence of measure valued functions&Definition~\ref{def: approximating boundary data}\\
$\WStarPath$&$\overstar{W}^{1,p}$ transport paths&Definition~\ref{def: approximating seq}\\
$\approxto$&Discrete approximating sequence of transport paths&Definition~\ref{def: approximating seq}\\
$\ELp^{\tau, p}_{\lambda}$&Energy on $\overstar{W}^{1,p}$ transport paths&Definition~\ref{def: cts system cost function}\\
\hline
\end{longtable}
}
\section{Notation and Background}\label{sec: background}
We begin by fixing some notation and recalling some basic definitions. All integrals of real-valued functions are Lebesgue integrals and $\mathcal{L}^d$ will denote Lebesgue measure on $\R^d$; we use the shorthand $dt$ for integration against $\mathcal{L}^1$. Let $\mathcal{B}(\domain)$ denote the Borel subsets of $\domain$. Then recall, given a Banach space $V$, a \emph{(finite) $V$-valued vector measure} is a set-valued map $\mu: \mathcal{B}(\domain)\to V$ satisfying
\begin{align*}
    \mu(\bigcup_{i\in \N} E_i)&=\sum_{i\in \N} \mu(E_i),\quad \{E_i\}_{i\in \N}\subset \mathcal{B}(\domain),\text{ pairwise disjoint},
\end{align*}
where the above sum is required to be absolutely convergent; we write $$\mathcal{M}(\domain; V):=\{\mu\mid \mu\text{ is a }V\text{-valued vector measure}\}.$$ We also let $\Lambda^m(\R^n)$ and $\Lambda_m(\R^n)$ denote the collections of $m$-cotangent vectors and $m$-tangent vectors respectively on $\R^n$, equipped with the norm induced by the Euclidean inner product. In this paper we will mainly consider the cases  $V=\Lambda_1(\R^n)$  or $V=\R$. Of course, since we are in the Euclidean case, it is possible to identify $\Lambda_1(\R^n)=\Lambda^1(\R^n)=\R^n$, however as we also utilize the language of currents later, we find it will be useful to keep this distinction at the price of a little more cluttered notation. We also recall the following norm on $\mathcal{M}(\domain; V)$ for a Banach space $V$.
\begin{defin}\label{def: TV norm}
If $\mu\in \mathcal{M}(\domain; V)$ for a Banach space $(V, \norm{\cdot}_V)$, the \emph{total variation (TV) norm} of $\mu$ is defined by
\begin{align*}
    \TV{\mu}:=\sup_P\sum_{A\in P}\norm{\mu(A)}_V,
\end{align*}
where the supremum is taken over all finite partitions $P$ of $\domain$ into Borel measurable sets, 
and $\mu$ is of \emph{bounded variation} if $\TV{\mu}<\infty$.
\end{defin}
For notational simplicity, we will write
\begin{align*}
    \meas:&=\mathcal{M}(\domain; \R),\qquad
    \vecmeas:=\mathcal{M}(\domain; \Lambda_1(\R^n)),
\end{align*}
with the understanding that each of these spaces is equipped with their respective total variation norm. It is well known that $\meas$  and $\vecmeas$ are Banach spaces, and their respective elements have bounded variation (see \cite[Remark 1.7]{AmbrosioFuscoPallara00}).

Next we write 
\begin{align*}
    C_0(\domain; \Lambda^1(\R^n)):&=
    \overline{C_c(\domain; \Lambda^1(\R^n))}^{\norm{\cdot}_{C(\domain)}}=\overline{\{\xi\in C(\domain; \Lambda^1(\R^n))\mid \spt \xi\Subset \domain\}}^{\norm{\cdot}_{C(\domain)}},
\end{align*}
where  $\norm{\cdot}_{C(\domain)}$ is the standard supremum norm; i.e. $C_c(\domain; \Lambda^1(\R^n))$ is the set of compactly supported, continuous vector fields and $C_0(\domain; \Lambda^1(\R^n))$ the set of continuous vector fields that decay to zero at infinity, both will be normed by $\norm{\cdot}_{C(\domain)}$. We will also write  $C_0(\R^d)$ and $C_c(\R^d)$ for the analogues with real-valued functions, $C^\infty_c(\R^d)$ for the set of smooth, compactly supported real-valued functions (with no particular norm), and $C_b(\R^d)$ for the set of bounded, continuous functions on $\R^d$ for any dimension $d$ (and obvious modifications for functions defined on subsets). 

For  $\xi=(\xi_1, \ldots, \xi_n)\in \covecfield$ and $\mu=(\mu^1, \ldots, \mu^n)\in \vecmeas$, the duality pairing is
\begin{align*}
 \inner{\mu}{\xi}=\int_{\domain} \xi(x)\cdot d\mu(x):=\sum_{i=1}^n  \int_{\domain } \xi_i(x)d\mu^i(x).
\end{align*}
It is then known that (see \cite[Proposition 1.47 and Remark 1.57]{AmbrosioFuscoPallara00}) 
\begin{align}\label{eqn: vector measures dual}
    (\vecmeas, \TV{\cdot})
    &=(\covecfield, \norm{\cdot}_{C(\domain)})^\ast,\qquad 
    (\meas, \TV{\cdot})
    =(C_0(\domain), \norm{\cdot}_{C(\domain)})^\ast,\\
    \TV{\mu}&=\sup\{\inner{\mu}{\xi}\mid \xi\in \covecfield,  \norm{\xi}_{C(\domain)}\leq 1\}\notag.
\end{align}
We also recall here the notion of weak convergence for sequences of Radon measures (with finite total variation).
\begin{defin}
    For $\{\mu_k\}_{k\in \N}\cup \{\mu\}\subset \meas$ we say \emph{$\mu_k$ weakly converges to $\mu$} and write $\mu_k\wto \mu$ if 
    \begin{align*}
        \lim_{k\to\infty}\int_{\domain}\phi d\mu_k=\int_{\domain}\phi d\mu,\qquad \forall \phi\in C_b(\domain).
    \end{align*}
\end{defin}
\begin{rem}
    Recall that a sequence in $\prob(\domain)$ which weak-$\ast$ converges (viewed as a sequence in $\meas$, in duality with $C_0(\domain)$) does not necessarily converge to an element of $\prob(\domain)$. However, if it is known that the limit belongs to $\prob(\domain)$, then the sequence actually converges weakly (see \cite[Lemma 2.1.13]{FigalliGlaudo23book}).
\end{rem}
Finally, we will denote weak-$\ast$ convergence of a sequence $\{F_i\}_{i\in \N}$ to $F$ in $V^\ast$ for some Banach space $V$ by $F_i \overset{V^\ast}{\wto} F$.

\subsection{Bochner integration theory}\label{subsec: bochner}
In this paper, we will mainly be interested in functions from the $[0, 1]$ interval into the Banach space $(\vecmeas, \TV{\cdot})$. To this end, we recall some of the basic definitions for measurability and Bochner integrals of vector-valued functions, following \cite{DiestelUhl77}. The phrase ``almost everywhere'' with no further qualifications will be with respect to Lebesgue measure on $[0, 1]$.
\begin{defin}\label{def: Lp space}
Let $(V, \norm{\cdot}_V)$ be a Banach space. A function $F: [0, 1]\to V$ is \emph{simple} if there exist $v_1,\ldots, v_K\in V$ and Borel measurable subsets $A_1,\ldots, A_K\subset [0, 1]$ such that $F[t]=\sum_{k=1}^K v_k \mathds{1}_{A_k}(t)$. $F$ is \emph{strongly measurable} if there exists a sequence of simple functions $\{F_i\}_{i\in \N}$ such that
\begin{align*}
    \lim_{i\to\infty} \norm{F_i[t]-F[t]}_V=0,\quad a.e.\ t\in [0, 1].
\end{align*} 
For a simple function $\sum_{k=1}^K v_i \mathds{1}_{A_k}(\cdot): [0, 1]\to V$, its \emph{Bochner integral} is defined by\\ $\int_0^1 \sum_{k=1}^K v_k \mathds{1}_{A_i}(t)dt:=\sum_{k=1}^K  v_k\mathcal{L}^1(A_k)\in V$.

A strongly measurable $F: [0, 1]\to V$ is \emph{Bochner integrable} if there exists a sequence of simple functions $\{F_i\}_{i\in \N}$ such that 
\begin{align*}
    \lim_{i\to\infty} \int_0^1\norm{F_i[t]-F[t]}_Vdt=0.
\end{align*} In that case, define 
\begin{align*}
    \int_0^1 F[t]dt:=\lim_{i\to\infty}\int_0^1 F_i[t]dt\in V
\end{align*}
as the Bochner integral of $F$.

Finally, for a strongly measurable $F: [0, 1]\to V$ and $1\leq p\leq \infty$, define the norm 
\begin{align*}
    \norm{F}_{L^p(V)}=
    \begin{cases}
    (\int_0^1 \norm{F[t]}_V^pdt)^{1/p},&1\leq p<\infty,\\
    \esssup_{[0,1]}\norm{F[t]}_V,&p=\infty,
    \end{cases}
\end{align*}
and the
\emph{Lebesgue-Bochner space}
\begin{align*}
    L^p(V):=\{F:[0, 1]\to V\text{ strongly measurable}\mid  \norm{F}_{L^p(V)}<\infty\},
\end{align*}
identifying functions that are equal almost everywhere on $[0, 1]$. We will write simply $L^p$ to denote the case $V=\R$. Since most of our Lebesgue-Bochner spaces will consist of strongly measurable functions with domain $[0, 1]$ equipped with $\mathcal{L}^1$, we will usually suppress this in the notation. However, we will have use later for Lebesgue-Bochner (and Sobolev versions, see below) spaces whose domains are another interval $I$ instead of $[0, 1]$, these are defined by an obvious generalization of the definition above, and we will always specify the domain in the notation, writing, e.g. $L^p(I; V)$ in this case. Finally, $L^p_{loc}(\R; V)$ denotes strongly measurable $V$-valued functions whose restriction to any compact interval $I$ belongs to $L^p(I; V)$.
\end{defin}

Our main setting will be on certain vector-valued \emph{Sobolev spaces}.
\begin{defin}
    If $V$ is a Banach space and $F\in L^1(V)$ recall the \emph{weak derivative} of $F$ on $[0, 1]$ is an $F'\in L^1(V)$ satisfying
    \begin{align*}
        \int_0^1\eta'(t)F[t]dt=-\int_0^1\eta(t)F'[t]dt
    \end{align*}
    for all $\eta\in C^\infty_c((0, 1))$.

    For $1\leq p\leq \infty$, recall the \emph{Bochner-Sobolev space} $W^{1, p}(V)$ is defined by
    \begin{align*}
        W^{1, p}(V):=\{F\in L^p(V)\mid \text{with weak derivative }F'\in L^p(V)\}.
    \end{align*}
\end{defin}
For the remainder of the paper, given $p\in [1, \infty]$, the number $p'\in [1, \infty]$ will denote its H\"older conjugate, satisfying $\frac{1}{p}+\frac{1}{p'}=1$. 
We will also need to consider a weaker notion of  Sobolev spaces for functions taking values in the dual of a Banach space.
\begin{defin}\label{def: weak star sobolev}
    If $1< p\leq \infty$, $V^\ast$ is the dual of a Banach space, and $F\in L^p(V^\ast)$, we define the \emph{$p$-weak-$\ast$ derivative} of $F$ on $[0, 1]$ as an element $\overstar{F}'\in L^{p'}(V)^\ast$ satisfying
    \begin{align*}
        \int_0^1\eta'(t)\inner{F[t]}{v}dt=-\inner{\overstar{F}'}{\eta v}
    \end{align*}
    for all $\eta\in C^\infty_c((0, 1))$ and $v\in V$. Here, the function $t\mapsto \eta(t)v$ is viewed as an element of $L^{p'}(V)$. 
    
    We will write $\overstar{W}^{1, p}(V^\ast)$ to denote the space of all functions in $L^p(V^\ast)$ with a $p$-weak-$\ast$ derivative.
\end{defin}
In the special case $V^\ast=\meas$, for $\compact\subset \domain$, by an abuse of notation, we will write
\begin{align}
    &\overstar{W}^{1, p}(\prob(K))\notag\\
    :=&\{F\in \overstar{W}^{1, p}(\meas)\mid F[t]\geq 0,\ F[t](\domain)=1,\ a.e.\ t\in [0, 1],\ \bigcup_{t\in [0, 1]}\spt F[t]\subset K\},\notag\\
    &\overstar{W}^{1, p}(\prob_c):=\bigcup_{K\subset \domain\ \text{compact}}\overstar{W}^{1, p}(\prob(K))\label{eqn: sobolev spaces}
\end{align}
(with the understanding that the union in the definition above is taken up to a set of $\mathcal{L}^1$-measure zero) and similarly for probability measure-valued Sobolev spaces defined below. If $K=\domain$, we will simply write $\overstar{W}^{1, p}(\prob)$.
\begin{rem}\label{rem: weak derivative remark}
    It is clear that $p$-weak-$\ast$ differentiation is linear and $W^{1, p}(V^\ast)\subset \overstar{W}^{1, p}(V^\ast)$, with $p$-weak-$\ast$ derivatives matching the usual weak derivative. Since $p>1$, approximating simple functions by linear combinations of functions of the form $\eta v$, it is not hard to see that $p$-weak-$\ast$ derivatives are also unique. 
\end{rem}

\begin{rem}\label{rem: radon nikodym and duals}
By \eqref{eqn: vector measures dual}, it can be seen that $L^p(\vecmeas)$ embeds isometrically as a subspace of\\ $L^{p'}(\covecfield)^\ast$ for $1< p\leq\infty$ (see \cite[Chapter IV, Section 1]{DiestelUhl77}). 
 In particular, the Banach-Alaoglu theorem applies to show that bounded subsets of $L^p(\vecmeas)$ are sequentially weak-$\ast$ compact. However, we note carefully here that $L^p(\vecmeas)\neq L^{p'}(\covecfield)^\ast$; this is because $\vecmeas$ does not have the Radon--Nikodym property, as it is a non-separable dual space of a separable Banach space (see \cite[Chapters III.3 and IV.1]{DiestelUhl77}). In fact, it is this failure of equality between $L^p(\vecmeas)$ and $ L^{p'}(\covecfield)^\ast$ which is the main source of difficulty.
\end{rem}
\begin{rem}\label{rem: weakly measurable}
    Suppose $1<p\leq \infty$, $V$ is a separable Banach space, and $\overstar{F}\in L^{p'}(V)^\ast$. Although $\overstar{F}$ may not be a strongly measurable map into $V^\ast$, there is an associated weak-$\ast$ measurable map in the following sense: since $1\leq p'<\infty$, we can apply \cite[Chapter II.13.5, Theorem 8]{Dinculeanu67book} with $T=[0, 1]$, $\mathscr{C}=\mathcal{B}([0, 1])$, $\nu=\mathcal{L}^1_{[0, 1]}$, $E=V$, $F=Z=\R$, and $U=\overstar{F}$ to find a map $H: [0, 1]\to V^\ast$ such that for any $\Psi\in L^{p'}(V)$, the real-valued function $t\mapsto \inner{H[t]}{\Psi[t]}$ is Borel measurable, 
    \begin{align*}
        \inner{\overstar{F}}{\Psi}
        =\int_0^1 \inner{H[t]}{\Psi[t]}dt,
    \end{align*}
    and for any interval $I\subset [0, 1]$,
    \begin{align}\label{eqn: weak measurable rep norm}
        \norm{\norm{H[\cdot]}_{V^\ast}\mathds{1}_I}_{L^p}=\norm{\overstar{F}\vert_{L^{p'}(I; V)}}_{L^{p'}(V)^\ast}.
    \end{align}
Note in the final equation above,  since $V$ is separable, as a countable supremum of Borel measurable functions, the function $t\mapsto \norm{H[t]}_{V^\ast}$ is also Borel measurable. 
\end{rem}
\section{Discrete transport paths}\label{sec: discrete}
Similar to \cite{QXia2003} and \cite{BrancoliniWirth18} in the time-independent case, we first define time-dependent paths connecting time-dependent atomic measures and their associated energies, then show such energies can be extended to the nonatomic case. For brevity we will omit the adjective ``time-dependent'' for the remainder of the paper.  

By an \emph{atomic probability measure} we will mean a Borel probability measure that is a finite linear combination of delta measures; we will denote this collection by  $\atomic$. An \emph{atomic probability measure-valued function} will be some $a: [0, 1]\to \atomic$ of the form 
\begin{align}\label{eqn: atomic measure}
    a[t]:=\displaystyle\sum_{j=1}^{J} a_j(t) \delta_{x_j}
\end{align} 
where $J\geq 1$, each $a_j: [0,1]\to [0,\infty)$ is Borel measurable, and each $\delta_{x_j}$ a Dirac measure supported at $x_j$ for some collection of points $\{x_j\}_{j=1}^J\subset \domain$; where we also  assume the mass condition
\begin{align}\label{eqn: total mass}
    \sum_{j=1}^J a_j(t)=1,\quad a.e.\ t\in [0, 1].
\end{align}
We emphasize, the weights of such a function are allowed to depend on time, but the points supporting the delta measures may not. 
By an abuse of notation, we will use $\spt a$ to denote the collection of points $\{x_j\}_{j=1}^J$ associated to $a$, and will write
\begin{align}
    L^p(\atomic):=\{a\in L^p(\meas)\mid a\text{ is an atomic probability measure-valued function}\},\notag\\
    W^{1, p}(\atomic):=\{a\in W^{1, p}(\meas)\mid a\text{ is an atomic probability measure-valued function}\},\notag\\
    \overstar{W}^{1, p}(\atomic):=\{a\in \overstar{W}^{1, p}(\meas)\mid a\text{ is an atomic probability measure-valued function}\}.\label{eqn: sobolev atomic}
\end{align}
We now define what is a \emph{transport path} between two atomic measure-valued functions.
\begin{defin}\label{def: discrete transport path}
Suppose $a^\pm$ are atomic measure-valued functions. A \emph{discrete transport path from $a^+$ to $a^-$}, is a weighted directed graph $G$ with finite vertex set $V(G)\subset \domain$, directed edge set $E(G)$, and a weight function
\begin{align*}
    w: E(G)\times[0,1]\to [0,\infty)
\end{align*}
such that $w(e, \cdot)$ is Borel measurable for each $e\in E(G)$ and such that the following properties hold: 
\begin{enumerate}[(i)]
    \item\label{eqn: boundary in vertices} $\spt a^+\cup \spt a^-\subseteq V(G)$,
    \item\label{eqn: Kirchhoff} Writing $e^-$ and $e^+$ for the terminal and initial endpoints of an edge $e\in E(G)$ and 
    \begin{align*}
    \In(y):=\{e\in E(G)\mid e^-=y\},\quad \Out(y):=\{e\in E(G)\mid e^+=y\},
\end{align*}
for any $y\in V(G)$,
\begin{align*}
    a^+[t](\{y\})+\sum_{e\in \In(y)}w(e, t)&=a^-[t](\{y\})+\sum_{e\in \Out(y)}w(e, t),\quad a.e.\ t\in [0, 1].
\end{align*}
\item\label{eqn: null overlap}For any $e\neq e'\in E(G)$, either $\HOne{e\cap e'}=0$, or $e'=-e$ and $w(e, t)w(-e, t)=0$ for a.e. $t\in [0, 1]$; here $-e$ is the edge such that $(-e)^+=e^-$ and $(-e)^-=e^+$.
\end{enumerate}
We will write $G\in \DPath(a^+, a^-)$ for short.
\end{defin}
We will write $\overrightarrow{y_1, y_2}$ for the directed edge from $y_1\in \domain$ to $y_2\in \domain$. Condition~\eqref{eqn: Kirchhoff} above states $G$ must satisfy a Kirchhoff condition at each vertex for a.e. $t\in [0, 1]$. 
We may interpret a discrete transport path $G$ as an $\vecmeas$-valued function on $[0, 1]$ by setting
\begin{align}\label{eqn: vector measure rep}
    G[t](E):=\sum_{e\in E(G)}w(e, t)\HOne{E\cap e}\vec{e}, \quad \forall t\in [0, 1]
\end{align}
where $\vec{e}$ is the unit length vector in the direction of edge $e$, and $\mathcal{H}^1$ is the $1$-dimensional Hausdorff measure. We write $[[e]]:=\vec e\mathcal{H}^1\vert_e$ and by an abuse of notation, denote the above discrete transport path by $G=\sum_{e\in E(G)}w(e, \cdot)[[e]]$. 
\begin{defin}
If $a^\pm\in W^{1, p}(\atomic)$ (resp. $\overstar{W}^{1, p}(\atomic)$) we say $G$ is a \emph{discrete $W^{1, p}$ transport path from $a^+$ to $a^-$} (resp. \emph{$\overstar{W}^{1, p}$ transport path}) if $G\in W^{1,p}(\vecmeas)\cap \DPath(a^+, a^-)$ (resp. $G\in \overstar{W}^{1,p}(\vecmeas)\cap \DPath(a^+, a^-)$) where we interpret $G$ in the sense of \eqref{eqn: vector measure rep}.
We will write $G\in W^{1, p}\mhyphen\Path_D(a^+, a^-)$ (resp. $G\in \overstar{W}^{1, p}\mhyphen\Path_D(a^+, a^-)$).

We make an analogous definition by replacing $W^{1, p}$ everywhere by $L^p$ above.
\end{defin}
Next we define a notion of energy for discrete (actually $L^p$) transport paths. Here, following \cite{BrancoliniWirth18}, we consider more general energies than the concave powers used in \cite{QXia2003}.

\begin{defin}\label{def: tau mass}
For a transportation cost $\tau$ and a rectifiable $m$-current $\tilde T=[[\tilde E, \tilde{V}, \tilde \theta]]$, the \emph{$\tau$-mass} of $\tilde T$ is defined by
\begin{align*}
    \Mcur^\tau(\tilde T):=\int_{\tilde E}\tau(\tilde\theta(x)) d\mathcal{H}^m(x).
\end{align*}
\end{defin}
Note that if $\tau(s)=\abs{s}$, then $\Mcur^\tau=\Mcur$ on rectifiable currents, which we use freely.

We also have use for the following known lemma, whose (quite elementary) proof we provide for completeness.
\begin{lem}[{\cite[proof of Theorem 5]{Laatsch62}, \cite[Lemma 1.3]{BrancoliniWirth18}}]\label{lem: subadditive bound}
    If $\tau$ is a transportation cost, for any $m>0$ let
    \begin{align*}
        \rho(\tau, m):=\inf_{w\in[m/2, m]}\frac{\tau(w)}{w}.
    \end{align*}
    Then for any $w\in [0, m]$, it holds that $\tau(w)\geq \rho(\tau, m)w$.
\end{lem}
Note that since $\tau>0$ on $(0, \infty)$ and is lower semi-continuous, $\rho(\tau, m)>0$ for any $m>0$.
\begin{proof}
    Fix $m>0$, suppose $w\in (0, m/2)$, and let $k\in \N$ be such that $kw\in[m/2, m]$. Then by  subadditivity of $\tau$,
    \begin{align*}
        \rho(\tau, m)
        &\leq \frac{\tau(kw)}{kw}
        \leq \frac{k\tau(w)}{kw}
        =\frac{\tau(w)}{w}.
    \end{align*}
\end{proof}

\begin{defin}\label{def: system cost function}
Suppose $1\leq p\leq \infty$, $G=\sum_{e\in E(G)}w(e, \cdot)[[e]]$ is a discrete $L^p$ transport path, $\lambda>0$, and $\tau$ is a transportation cost. The \emph{$\MLp_p^\tau$ cost of $G$} is defined as
\begin{align}\label{eqn: system cost}
    \MLp_p^\tau (G)
    :&= \Lp{\sum_{e\in E(G)}\tau(\abs{w(e, \cdot)})\HOne{e}}
    =\Lp{\Mcur^\tau(G[\cdot])}.
\end{align}
\end{defin}
In the case when $\tau(x)=\abs{x}^\alpha$ for some $\alpha\in \R$, we will write $\MLp_p^\alpha$ for $\MLp_p^\tau$. Note that
\begin{align*}
    \MLp_p^1(G)=\norm{G}_{L^p(\vecmeas)},
\end{align*}
a fact that we will use freely.
\subsection{Properties of discrete transport paths}\label{subsec: properties discrete paths}
We now show some properties of discrete $W^{1, p}$ transport paths, before defining the energy we are interested in minimizing.
\begin{lem}\label{lem: Lp weights}
If $1<p\leq \infty$, a discrete transport path given by $G=\sum_{e\in E(G)}w(e, \cdot)[[e]]$ belongs to $W^{1,p}(\vecmeas)$, if and only if it belongs to $\overstar{W}^{1,p}(\vecmeas)$, if and only if each weight function $w(e, \cdot)\in W^{1,p}$; additionally in all three cases $\overstar{G}'=G'=\sum_{e\in E(G)}w'(e, \cdot)[[e]]$ where $w'(e, \cdot)$ denotes the weak derivative of $w(e, \cdot)$. An analogous statement holds for $W^{1, p}(\atomic)$ and $\overstar{W}^{1, p}(\atomic)$.
\end{lem}
\begin{proof}
    Since $W^{1,p}(\vecmeas)\subset \overstar{W}^{1,p}(\vecmeas)$, we first suppose $G\in \overstar{W}^{1,p}(\vecmeas)$. Then for any edge $e\in E(G)$,
\begin{align*}
    \HOne{e}^p\norm{w(e, \cdot)}_{L^p}^p&=\int_0^1\TV{w(e, t)[[e]]}^pdt\leq \int_0^1 \TV{G[t]}^pdt= \norm{G}_{L^p(\vecmeas)}^p,
\end{align*}
hence each $w(e, \cdot)\in L^p$. 

Next, fix an edge $e_0\in E(G)$, and let $\xi\in \covecfield$ be such that $\inner{[[e_0]]}{\xi}\neq 0$ and $\spt\xi$ does not intersect any edge in $E(G)\setminus \{e_0\}$ except possibly $-e_0$. Let $\overstar{G}'\in L^{p'}(\covecfield)^\ast$ be the $p$-weak-$\ast$ derivative of $G$. For any $\eta\in C^\infty_c((0, 1))$, we have
\begin{align*}
    \abs{\inner{\overstar{G}'}{\eta \xi}}
    &\leq \norm{\overstar{G}'}_{L^{p'}(\covecfield)^\ast}\norm{\eta \xi}_{L^{p'}(\covecfield)}\\
    &\leq \norm{\overstar{G}'}_{L^{p'}(\covecfield)^\ast}\norm{\xi}_{ \covecfield}\norm{\eta}_{L^{p'}},
\end{align*}
thus we can extend the map $\eta\mapsto \inner{\overstar{G}'}{\eta \xi}$ to an element of $(L^{p'})^\ast$. Since $p>1$ we have $p'<\infty$, thus by Riesz representation, there exists a function $g\in L^p$ such that for any $\eta\in C^\infty_c((0, 1))$ we have
\begin{align*}
    \int_0^1\eta(t)g(t)dt
    &= \inner{\overstar{G}'}{\eta \xi}
    =-\int_0^1\eta'(t)\inner{G[t]}{\xi}dt\\
    &=\inner{[[e_0]]}{\xi}\int_0^1\eta'(t)(w(-e_0, t)-w(e_0, t))dt.
\end{align*}
Hence $w(e_0, \cdot)-w(-e_0, \cdot)$ is weakly differentiable on $[0, 1]$, and its weak derivative is $g/\inner{[[e_0]]}{\xi}$. Taking positive and negative parts and recalling Definition~\ref{def: discrete transport path}~\eqref{eqn: null overlap},  we have $w(\pm e_0, \cdot)\in W^{1,p}$.

Now suppose for each $e\in E(G)$, the weight function $w(e, \cdot)\in W^{1, p}$. Then for any $\xi\in \covecfield$, we see $t\mapsto \inner{G[t]}{\xi}=\sum_{e\in E}w(e, t)\inner{[[e]]}{\xi}$ is Borel measurable on $[0, 1]$. The subset $\{\sum_{e\in E}q_e[[e]]\mid q_e\in \Q,\ \forall e\in E\}$ is countable and dense (with respect to $\TV{\cdot}$) in the image $G([0, 1])$. Since $\covecfield$ is a norming set for $\vecmeas$ in the sense of \cite[Chapter II.1, Corollary 4]{DiestelUhl77}, by the same corollary we see $G$ is strongly measurable, and likewise for $\sum_{e\in E}w'(e, \cdot)[[e]]$. Since each $w(e, \cdot)\in W^{1, p}$ we have
\begin{align*}
    \norm{G}_{L^p(\vecmeas)}^p
    &\leq (\sharp E(G))^{p-1}\sum_{e\in E(G)}\int_0^1\TV{w(e, t)[[e]]}^pdt\\
    &=(\sharp E(G))^{p-1}\sum_{e\in E(G)}\HOne{e}^p\norm{w(e, \cdot)}_{L^p}^p<\infty,
\end{align*}
hence $G\in L^p(\vecmeas)$, and similarly $\sum_{e\in E}w'(e, \cdot)[[e]]\in L^p(\vecmeas)$. Finally, for any $\eta\in C^\infty_c((0, 1))$ we calculate
\begin{align*}
    \int_0^1 \eta(t)\sum_{e\in E(G)}w'(e, t)[[e]]dt
    &=\sum_{e\in E(G)}\left(\int_0^1\eta(t)w'(e, t)dt\right)[[e]]\\
    &=-\sum_{e\in E(G)}\left(\int_0^1\eta'(t)w(e, t)dt\right)[[e]]
    =-\int_0^1 \eta'(t)G[t]dt,
\end{align*}
thus $G'=\sum_{e\in E}w'(e, \cdot)[[e]]$, hence $G\in W^{1, p}(\vecmeas)$. This finishes the proof of the equivalences, and the above calculation along with uniqueness of weak derivatives yields the claimed representation of $\overstar{G}'=G'$.

The proof for $W^{1, p}(\atomic)$ and $\overstar{W}^{1, p}(\atomic)$ is essentially identical, replacing $\xi$ above by a compactly supported real-valued continuous function instead. 
\end{proof}
\begin{rem}
    By the above lemma, we see that the weights $w(e, \cdot)$ in a discrete $W^{1, p}$ transport path are (absolutely) continuous on $[0, 1]$. For the remainder of this paper, we will assume the weight functions $w(e, \cdot)$ are taken as the (unique) absolutely continuous versions for any discrete $W^{1, p}$ transport path. A similar comment holds for weights of elements of $W^{1, p}(\atomic)$.

    Also, the lemma shows that a discrete $W^{1,p}$ transport path can be viewed as a polyhedral $1$-chain with coefficients in $W^{1,p}$.
\end{rem}
\subsection{Currents}\label{subsec: currents}
In order to prove existence of minimizers for our energy, we will also need to make use of the theory of currents. We recall some relevant definitions and tools here, see for example, \cite[Chapter 7]{KrantzParks08} for details. Given any open domain $D\subset \R^d$ for some dimension $d$ and any $m\in \Z_{\geq 0}$, we write $\mathcal{D}^m(D)$ for the space $C_c^\infty(D; \Lambda^m(\R^n))$  of \emph{compactly supported smooth differential $m$-forms on $D$} equipped with the usual Frech\'et topology, and $\mathcal{D}_m(D)$ for the \emph{space of $m$-currents} dual to $\mathcal{D}^m(D)$. \begin{defin}\label{def: mass and flat norm of currents}
If $D\subset \R^d$ is open, $
\tilde\psi=\sum_{1\leq j_1<\ldots<j_m\leq d} \tilde\psi_{j_1, \ldots, j_m}dx^{j_1}\wedge\ldots\wedge dx^{j_m}\in \mathcal{D}^m(D)$, and $\tilde T\in \mathcal{D}_m(D)$, the \emph{exterior derivative} $d\tilde\psi\in \mathcal{D}^{m+1}(D)$ of $\tilde \psi$, \emph{boundary} $\partial \tilde T\in \mathcal{D}_{m-1}(D)$, \emph{mass} $\Mcur(\tilde T)$, and \emph{flat norm} $\Fcur(\tilde T)$ of $\tilde T$ are respectively
\begin{align*}
d\tilde\psi:&=\sum_{j=1}^d \sum_{1\leq j_1<\ldots<j_m\leq d}\partial_{x^j}\tilde\psi_{j_1, \ldots, j_m}dx^j\wedge dx^{j_1}\wedge\ldots\wedge dx^{j_m},\\
    \partial \tilde T(\tilde \phi):&=
    \begin{cases}
        \tilde T(d\tilde \phi), &m\geq 1,\\
        0,&m=0,
    \end{cases}
    \quad\forall \tilde \phi\in \mathcal{D}^{m-1}(D),\\
    \Mcur(\tilde T):&=\sup\{\tilde T(\tilde \phi)\mid \tilde \phi\in \mathcal{D}^m(D),\ \norm{\tilde\phi}_{C(\domain)}\leq 1\},\\
    \Fcur(\tilde T):&=\sup\{\tilde T(\tilde \phi)\mid\tilde \phi\in \mathcal{D}^m(D),\ \max(\norm{\tilde\phi}_{C(\domain)}, \norm{d\tilde\phi}_{C(\domain)})\leq 1\}.
\end{align*}
\end{defin}
For a sequence $\{\tilde{T}_k\}_{k\in \N}\subset \mathcal{D}_m(D)$ and $\tilde T\in \mathcal{D}_m(D)$, we will write 
\begin{align*}
    \tilde{T}_k\flatto\tilde{T}\iff \lim_{k\to\infty}\Fcur(\tilde{T}_k-\tilde{T})=0,
\end{align*}
and for weak-$\ast$ convergence in duality with $\mathcal{D}^m(D)$ we write $\tilde{T}_k\currentto\tilde{T}$.
\begin{defin}\label{def: m-rectifiable sets}
    A set $E\subset \R^n$ is \emph{countably $m$-rectifiable} if there exist a countable collection $\{E_i\}_{i\in \N}$ of bounded subsets of $\R^m$, Lipschitz maps $f_i: E_i\to \R^n$, and an $\mathcal{H}^m$-null set $E_0\subset \R^n$ such that
    \begin{align*}
        E=E_0\cup \bigcup_{i=1}^\infty f_i(E_i).
    \end{align*}
\end{defin}
\begin{defin}\label{def: integral currents}
If $D\subset \R^d$ is open, $\tilde T\in \mathcal{D}_m(D)$ is a \emph{rectifiable $m$-current} if there is an $\mathcal{H}^m$-measurable, countably $m$-rectifiable set $\tilde E\subset D$, an $\mathcal{H}^m$-measurable map $\tilde{V}:\tilde{E}\to \Lambda_m(\R^d)$ where $\tilde{V}(x)$ is unit length and orients the approximate tangent space (see \cite[Definition 5.4.4]{KrantzParks08}) of $\tilde E$ at $x$ for $\mathcal{H}^m$-a.e. $x\in \tilde E$, and a nonnegative $\tilde\theta \in L^1_{loc}(\mathcal{H}^m\vert_{\tilde E}; \R)$ such that for any $\tilde\phi\in \mathcal{D}^m(D)$,
\begin{align*}
    \tilde T(\tilde\phi)=\int_{\tilde E}\tilde\theta(x)\inner{\tilde{V}(x)}{\tilde \phi(x)}d\mathcal{H}^m(x).
\end{align*}
In this case, we will write $\tilde T=[[\tilde E, \tilde{V}, \tilde \theta]]$.
\end{defin}
Later we will need to convert elements of $L^p(\vecmeas)$ to $2$-currents on a higher dimensional space.
\begin{defin}\label{def: map to two forms}
For $\tilde\phi\in \mathcal{D}^2(\R\times \domain)$ with
\begin{align*}
    \tilde\phi=\sum_{j=1}^n \tilde\phi_{tj}(t, x)dt\wedge dx^j+\sum_{j'<j}^n\tilde\phi_{j'j}(t, x)dx^{j'}\wedge dx^j,
\end{align*}
we define $\Lqform(\tilde\phi)\in L^{p'}(\R; \covecfield)$ by 
\begin{align*}
    \Lqform(\tilde\phi)(t):=\sum_{j=1}^n \tilde\phi_{tj}(t, \cdot)dx^j, \qquad t\in \R.
\end{align*}
Given a discrete transport path $G=\sum_{e\in E(G)}w(e, \cdot)[[e]]\in W^{1, p}(\vecmeas)$, by Lemma~\ref{lem: Lp weights} we can extend the coefficients $w(e, \cdot)$ to elements of $W^{1, p}(\R)$, compactly supported on $[-1, 2]$ (see \cite[Section 5.4]{Evans10book}), and thus identify $G$ with an element of $W^{1, p}(\R; \vecmeas)\subset L^p(\R; \vecmeas)$. Then we define 
$\twocurrent(G): \mathcal{D}^2(\R\times \domain)\to \R$
by letting
\begin{align*}
    \inner{\twocurrent(G)}{\tilde\phi}&=\inner{G}{\Lqform(\tilde\phi)}.
\end{align*}
\end{defin}
We can actually say a little bit more about $\twocurrent(G)$.
\begin{rem}\label{rem: extensions}
    If $G\in \WDPath(a^+, a^-)$ for $a^\pm\in W^{1, p}(\atomic)$, we can similarly extend $a^\pm$ to all of $\R$. Since the Sobolev extension operator is linear, the a.e. Kirchhoff conditions Definition~\ref{def: discrete transport path}~\eqref{eqn: Kirchhoff} continue to hold for the extended $G$ and $a^\pm$. Additionally from the boundedness of the extension operator, we have $\norm{a^\pm}_{L^1(\R)}<C$ for some $C>0$ independent of $a^\pm$.
\end{rem}
\begin{lem}
     For any discrete transport path $G\in W^{1, p}(\vecmeas)$, we have  
    $\twocurrent(G)\in \mathcal{D}_2(\R\times \domain)$.
\end{lem}

\begin{proof}
By the linearity of $G$, it is clear that $\twocurrent(G)$ is also linear.  Now suppose 
$\{\tilde{\phi}_i\}_{i\in \N}\subset \mathcal{D}^2(\R\times \domain)$
converges in the Fr\'echet topology to some $\tilde{\phi}$, this implies there are compact 
$I\subset \R$
and $K\subset\domain$ such that $\bigcup_{i\in \N} \spt \tilde{\phi}_i\cup \spt \tilde{\phi}\subset I\times K$ and $\norm{\tilde{\phi}_i-\tilde{\phi}}_{C(I\times K)}\to 0$ (see \cite[Section 9.1]{Folland99}).  If
\begin{align*}
    \tilde\phi&=\sum_{j=1}^n \tilde\phi_{tj}(t, x)dt\wedge dx^j+\sum_{j'<j}^n\tilde\phi_{j'j}(t, x)dx^{j'}\wedge dx^j,\\
    \tilde\phi_i&=\sum_{j=1}^n \tilde\phi_{tj, i}(t, x)dt\wedge dx^j+\sum_{j'<j}^n\tilde\phi_{j'j, i}(t, x)dx^{j'}\wedge dx^j,
\end{align*}
then for some $C_n>0$,
\begin{align*}
    \norm{\Lqform(\tilde \phi_i)-\Lqform(\tilde \phi)}_{L^{p'}(\R; \covecfield)}
    &\leq  C_n\left(\int_{I}\norm{\sum_{j=1}^n\abs{\tilde{\phi}_{tj, i}(t, \cdot)-\tilde{\phi}_{tj}(t, \cdot)}}_{C(K)}^{p'}dt\right)^{\frac{1}{{p'}}}\\
    &\leq C_n\norm{\tilde{\phi}_i-\tilde{\phi}}_{C(I\times K)}\to 0,
\end{align*}
thus
\begin{align*}
    \lim_{i\to\infty}\inner{\twocurrent(G)}{\tilde{\phi}_i}&=\lim_{i\to\infty}\inner{G}{\Lqform(\tilde\phi_i)}=\inner{G}{\Lqform(
\tilde\phi)}
=\inner{\twocurrent(G)}{\tilde{\phi}},
\end{align*}
showing that 
$\twocurrent(G)\in \mathcal{D}_2(\R\times \domain)$.
\end{proof}
\begin{rem}\label{rem: vector measures and 1 currents}
    It is well-known (see \cite[Chapter 7, Currents Representable by Integration]{KrantzParks08}) that the subset of $\mathcal{D}_1(\domain)$ with finite mass $\Mcur$ can be identified with $\vecmeas$. By this we mean that if $\tilde{T}\in \mathcal{D}_1(\domain)$ with $\Mcur(\tilde{T})<\infty$, it can be extended to all of $\covecfield$ as an element of $\vecmeas=(\covecfield)^\ast$, while any element of $\vecmeas$ can be restricted to $\mathcal{D}^1(\domain)$ to yield an element of $\mathcal{D}_1(\domain)$. We will freely use this identification throughout.

    Moreover, (see \cite[Section 7]{FedererFleming60}) it is known that for any bounded domain $D$ and $\{\tilde{T}_k\}_{k\in\N}\cup \{\tilde{T}\}\subset \mathcal{D}_1(D)$, the convergence $\tilde{T}_k\flatto \tilde{T}$ implies  $\tilde{T}_k\currentto \tilde{T}$, and if $\sup_{k\in \N}(\Mcur(\tilde{T}_k)+\Mcur(\partial \tilde{T}_k))<\infty$, the converse holds.
\end{rem}
\begin{rem}\label{rem: boundary of transport path}
    If $G = \sum_{e\in E(G)}w(e, \cdot)[[e]]\in \WDPath(a^+, a^-)$, then for a.e. $t\in [0, 1]$ and any $\phi \in \mathcal{D}^0(\domain)$ we have (using the a.e. Kirchhoff conditions on $G$)
    \begin{align*}
        \inner{\partial G[t]}{\phi}&=\inner{G[t]}{d\phi}
    =\sum_{e\in E(G)}\int_ew(e, t)\inner{[[e]]}{d\phi(x)}d\mathcal{H}^1(x) \\
    &=\sum_{y\in V(G)}\phi(y)\left(\sum_{e\in \In(y)}w(e, t)-\sum_{e\in \Out(y)}w(e, t)\right)\\
    &=\inner{a^-[t]-a^+[t]}{\phi},
    \end{align*}
    hence $\partial G[t]=a^-[t]-a^+[t]$. A similar calculation yields $\partial G'[t]=(a^-)'[t]-(a^+)'[t]$.

     Then, if  $\tilde\phi=\tilde\phi_tdt+\sum_{i=1}^n\tilde\phi_i dx^i\in \mathcal{D}^1(\R\times \domain)$ with 
    $\norm{\tilde\phi}_{C(\R\times\domain)}\leq 1$
    we have
    \begin{align*}
        \abs{\inner{\partial \twocurrent(G)}{\tilde\phi}}
        &=\abs{\inner{G}{\Lqform(d\tilde \phi)}}\\
        &=\abs{\int_{\R}\sum_{e\in E(G)}\sum_{i=1}^n\int_ew(e, t)(\partial_t\tilde\phi_i-\partial_{x^i}\tilde\phi_t)\inner{[[e]]}{dx^i}d\mathcal{H}^1(x) dt}\\
        &\leq \abs{\sum_{e\in E(G)}\int_{-1}^{2} w'(e, t)\sum_{i=1}^n\int_e\tilde\phi_i(t, x)\inner{[[e]]}{dx^i}d\mathcal{H}^1(x)dt}\\
        &+\abs{\int_{-1}^{2}\inner{a^+[t]-a^-[t]}{\tilde\phi_t(t, \cdot)}dt}\\
        &\leq \sqrt{n}\norm{G'}_{L^p(\vecmeas)}+C;
    \end{align*}
    here we have used Remark~\ref{rem: extensions} for the last two lines. In particular for some universal constant $C>0$ we have
    \begin{align*}
        \Mcur(\partial\twocurrent(G))\leq C(1+\norm{G'}_{L^p(\vecmeas)}).
    \end{align*}
\end{rem}

\begin{prop}\label{prop: discrete energies finite}
If $\tau$ is an admissible transportation cost, then $\MLp_p^\tau (G)<\infty$ for any discrete $L^p$ transport path $G$.
\end{prop}
\begin{proof}
Let $\beta$ be the concave function in the definition of admissibility, and $t_0>0$ be a point where $\beta$ is differentiable. Then we find that
\begin{align*}
    \MLp_p^\tau (G) &=\Lp{\sum_{e\in E(G)}\tau(w(e,\cdot)) \HOne{e}} \\
    &\leq \sum_{e\in E(G)} \Lp{\beta(w(e,\cdot))} \HOne{e} \\
    &\leq \sum_{e\in E(G)} \Lp{\beta(t_0)+\beta'(t_0)(w(e,\cdot)-t_0)} \HOne{e}\\
    &\leq \sum_{e\in E(G)} \left(\beta(t_0)+\beta'(t_0)(\Lp{w(e,\cdot)}+t_0)\right) \HOne{e}<\infty,
\end{align*}
using Lemma~\ref{lem: Lp weights} for the last line.
\end{proof}

We now discuss the concept of \emph{cycles} in discrete transport paths, which plays a role in defining the energy we will consider.
\begin{defin}\label{def: instantaneous cycle}
Let $G$ be a discrete $L^p$ transport path 
and suppose there exist a collection of distinct vertices $\{y_j\}_{j=1}^J\subset V(G)$ and edges $\{e_j\}_{j=1}^J\subset E(G)$ where $e_j=\overrightarrow{y_j, y_{j+1}}$ for $j=1, \ldots, J-1$  and $e_J=\overrightarrow{y_J, y_1}$. If $C$ is the ordered list $e_1, \ldots, e_J$, recall $C$ is said to be a \emph{cycle}; when we say $C\subset E(G)$ it will be understood that $C$ is considered as an ordered list. If $\min_{e_j\in C}w(e_j, t)>0$ for some $t\in [0, 1]$, 
we say \emph{$G$ has a cycle (supported on the collection of edges $C$) at $t$}. We say $G$ is \emph{never cyclic} if its has no cycles for any $t\in [0, 1]$.
\end{defin}
Note that a $G$ could be never cyclic, but the directed graph underlying it could contain a directed cycle: certain combinations of edge weights could be zero at various times to prevent any strong cycles. Also technically, we should consider when $G$ is ``almost never cyclic,'' however since all discrete transport paths we consider will be $W^{1, p}$ transport paths and we identify the edge weight functions with their continuous versions, $G$ is never cyclic if and only if it has no strong cycles for a.e. $t\in [0, 1]$. 
We now consider a decomposition of a discrete transport path into ``cyclical'' and its ``acyclical'' parts. 
\begin{defin}\label{def: cycle decomposition}
    Let $G=\sum_{e\in E(G)}w(e, \cdot)[[e]]$ be a discrete $W^{1, p}$ transport path, $\{C_i\}_{i=1}^I$ be all of the directed cycles contained in $E(G)$ and
     $\sigma\in S_I$ (here $S_I$ is the set of all permutations of $\{1, \ldots, I\}$).
     Now iteratively define the nonnegative 
    functions $w_{\sigma, 0}(e, \cdot)$, $W_{\sigma, i}\in W^{1, p}$ for $e\in E(G)$ and $i=1, \ldots, I$ by
    \begin{align*}
        W_{\sigma, 1}(t):&=\min_{e\in C_{\sigma(1)}}w(e,t)\\
        W_{\sigma, i+1}(t):&=\min_{e\in C_{\sigma(i+1)}}\left(w(e, t)-\sum_{i'=1}^i W_{\sigma, i'}(t)\mathds{1}_{C_{\sigma(i')}}(e)\right),\ i=1, \ldots, I-1,\\
        w_{\sigma, 0}(e, t):&=w(e, t)-\sum_{i=1}^I W_{\sigma, i}(t)\mathds{1}_{C_{\sigma(i)}}(e).
    \end{align*}
    Then we can define
    \begin{align*}
        G_{\sigma, 0}:&=\sum_{e\in E(G)}w_{\sigma, 0}(e, \cdot)[[e]], \quad
        G_{\sigma, i}:=\sum_{e\in C_{\sigma(i)}}W_{\sigma, i}[[e]],\ i=1, \ldots, I
    \end{align*}
    and we say $\{G_{\sigma, 0}\}\cup \{G_{\sigma, i}\}_{i=1}^I$ is \emph{the decomposition of $G$ with respect to $\sigma$}.
\end{defin}
Effectively, we want to count the derivatives of the weights of edges belonging to multiple cycles or a cycle and the ``acyclic part'' of a discrete transport path as contributing energy separately (for more discussion, see subsection~\ref{subsec: model choice} below). However, the order in which the directed cycles are ordered will change the resulting decomposition. With this in mind, we define the energy of a discrete transport path in the following manner.
\begin{defin}\label{def: discrete energy}
Given $\lambda\geq0$, $1< p\leq\infty$, a transportation cost $\tau$, and a discrete $W^{1, p}$ transport path $G$, we define the  energy $\DELp^{\tau, p}_{\lambda}$ of $G$ by
\begin{align*}
    \DELp^{\tau, p}_{\lambda}(G):&=\MLp_{p}^\tau(G)+\lambda\max_{\sigma\in S_I}\left(\norm{G'_{\sigma, 0}}_{L^p(\vecmeas)}+\sum_{i=1}^I\norm{G'_{\sigma, i}}_{L^p(\vecmeas)}\right),
\end{align*}
where $I$ is the number of directed cycles in $E(G)$.
\end{defin}
Having defined the relevant energy as above for discrete transport paths, we now show a quick lemma regarding the elimination of cycles.
\begin{lem}\label{lem: eliminating cycles}
    Suppose $\tau$ is a transportation cost, $1< p\leq \infty$, and $\lambda\geq0$, and $G\in \WDPath(a^+, a^-)$ for some $a^\pm\in W^{1, p}(\atomic)$. 
    \begin{enumerate}[(1)]
        \item \label{eqn: never cyclic energy}
        If $G$ is never cyclic, then 
        \begin{align*}
            \DELp^{\tau, p}_\lambda(G)=\MLp^\tau_p(G)+\lambda \norm{G'}_{L^p(\vecmeas)}.
        \end{align*}
        \item \label{eqn: reducing energy}
        If $\spt a^+\cap \spt a^-=\emptyset$, there exists $\tilde{G}\in \WDPath(a^+, a^-)$ that is never cyclic, such that $\DELp^{\tau, p}_\lambda(\tilde{G})\leq \DELp^{\tau, p}_\lambda(G)$.
    \end{enumerate}
\end{lem}
\begin{proof}
    First let $\{C_i\}_{i=1}^I$ be the collection of directed cycles in $E(G)$, and $\sigma_0\in S_I$ be the permutation which achieves the maximum value in the definition of $\DELp^{\tau, p}_\lambda(G)$. If $G$ is never cyclic, we can see that $W_{\sigma_0, i}\equiv 0$ for $i=1, \ldots I$, thus~\eqref{eqn: never cyclic energy} is immediate. 
    
    Next letting
    \begin{align*}
        \tilde{G}:=\sum_{e\in E(G)}w_{\sigma_0, 0}(e, \cdot)[[e]],
    \end{align*}
    since each $C_i$ is a directed cycle and $\spt a^\pm$ do not share any points, we can see that $\tilde{G}\in \WDPath(a^+, a^-)$, and by definition of $w_{\sigma_0, 0}$ it follows that $\tilde{G}$ is never cyclic. Since $0\leq w_{\sigma_0, 0}(e, \cdot)\leq w(e, \cdot)$ for all $e\in E(G)$, combining with \eqref{eqn: never cyclic energy} yields \eqref{eqn: reducing energy}.
\end{proof}

We finish this section by analyzing the most basic toy example.
\begin{example}\label{ex: toy example}
    Let $\tau(s)=\abs{s}^{1/2}$, $p=2$, and $\rho(t)=\sin^2(\pi(t+1/4))$, fix $P_\pm:=(\pm 1, k)\in \R^2$ with $k>1$, and define $a^+[t]:=(1-\rho(t))\delta_{P_-}+\rho(t)\delta_{P_+}$, $a^-:=\delta_0$; clearly $a^\pm\in W^{1, 2}(\atomic)$. To minimize $\DELp^{\tau, 2}_{\lambda}$ over $\WDPath(a^+, a^-)$, by Lemma~\ref{lem: eliminating cycles}~\eqref{eqn: reducing energy} we may restrict the problem to $G$ that are never cyclic. Since $\tau$ is a concave power, by a construction analogous to~\cite[Proposition 2.1]{QXia2003}, for a fixed $t\in [0, 1]$, it can be shown that $G[t]$ is a tree, hence $\sharp E(G[t])=\sharp V(G[t])-1$. Summing the degrees of the vertices, we can show that 
    \begin{align*}
        G[t]=(1-\rho(t))[[\overrightarrow{P_-, x(t)}]]+\rho(t)[[\overrightarrow{P_+, x(t)}]]+[[\overrightarrow{x(t), 0}]]
    \end{align*}
    for some $x(t)\in \R^2$. However, since $G$ is supported on a finite set of edges and each edge weight is continuous by Lemma~\ref{lem: Lp weights}, $x(t)$ must be constant. Using that $\rho(t)=1-\rho(1-t)$ for $t\in [0, 1/2]$, we may restrict to the case $x(t)\equiv (0, y)$ for some $y\in \R$, thus writing $g(y):=1+(y-k)^2$, the problem is reduced to minimizing
    \begin{align*}
        f_{\lambda}(y):&=\sqrt{\int_0^1\abs{\left[\sqrt{1-\rho(t)}+\sqrt{\rho(t)}\right]\sqrt{g(y)}+\abs{y}}^2dt}\\
        &+\lambda \sqrt{\int_0^1\abs{2\abs{\rho'(t)}\sqrt{g(y)}}^2dt}\\
        &=\sqrt{\int_0^1\abs{\left[\abs{\cos(\pi(t+1/4))}+\abs{\sin(\pi(t+1/4))}\right]\sqrt{g(y)}+\abs{y}}^2dt}\\
        &+\lambda \sqrt{\int_0^1\abs{2\abs{2\pi\sin(\pi(t+1/4))\cos(\pi(t+1/4))}\sqrt{g(y)}}^2dt}\\
        &=2\sqrt{\int_0^{1/4} \abs{\sqrt{2}\cos(\pi t)\sqrt{g(y)}+\abs{y}}^2dt}
        +2\pi\lambda \sqrt{g(y)}\sqrt{\int_0^1\cos^2(2\pi t)dt} \\
        &=2\sqrt{2g(y)\int_0^{1/4} \cos^2(\pi t)dt
        +2\sqrt{2}\abs{y}\sqrt{g(y)}\int_0^{1/4} \cos(\pi t)dt
        +\frac{y^2}{4}}\\
        &+2\pi\lambda \sqrt{g(y)}\sqrt{\int_0^1\cos^2(2\pi t)dt} \\
        &=2\sqrt{2g(y)\frac{\pi+2}{8\pi}
        +\frac{2\abs{y}\sqrt{g(y)}}{\pi}
        +\frac{y^2}{4}}
        +\sqrt{2}\pi\lambda \sqrt{g(y)}\\
        &= \sqrt{y^2+ \frac{8}{\pi}\abs{y}\sqrt{g(y)} + \left(\frac{2}{\pi}+1\right)g(y)}
        +\sqrt{2}\pi\lambda\sqrt{g(y)}.
    \end{align*}
    By the properties of $g$, it can be seen $f_\lambda$ is strictly convex and attains its minimum on $[0, \infty)$. The right sided derivative of $f_\lambda$ at $0$ is
    \begin{align*}
        f'_\lambda(0^+)=\frac{4\sqrt{1+k^2}-(2+\pi)k}{\sqrt{\pi(2+\pi)(1+k^2)}}-\frac{\sqrt{2}\pi\lambda }{\sqrt{1+k^2}},
    \end{align*}
    \begin{wrapstuff}[l,width=0.45\textwidth,type=figure]
    \resizebox{0.95\linewidth}{!}{
\begin{tikzpicture}
\tikzset{->-/.style={decoration={
  markings,
  mark=at position .5 with {\arrow{>}}},postaction={decorate}}}
\tikzstyle{every node}=[font=\fontsize{18.2pt}{23.7pt}\selectfont]
\draw[->-] (1.5,14) to (2.75,11.5);
\draw[->-] (4,14) to (2.75,11.5);
\filldraw [black] (1.5,14) circle (2pt);
\filldraw [black] (4,14) circle (2pt);
\filldraw [black] (2.75,11.5) circle (2pt);
\draw[->-] (5.25,14) to (6.5,12.75);
\draw[->-] (7.75,14) to (6.5,12.75);
\draw[->-] (6.5,12.75) to (6.5,11.5);
\filldraw [black] (5.25,14) circle (2pt);
\filldraw [black] (7.75,14) circle (2pt);
\filldraw [black] (6.5,11.5) circle (2pt);
\end{tikzpicture}
}
\caption{Left: time-dependent minimizer.\\ Right: stationary minimizer at $t=0$.}
\label{fig:toy paths}
    \end{wrapstuff}
\noindent for $\lambda\leq (4\sqrt{1+k^2}-(2+\pi)k)/k\pi\sqrt{2\pi(2+\pi)}$ the derivative is positive,
    hence $f_\lambda$ attains its minimum when $y=0$ (illustrated in Figure~\ref{fig:toy paths}). The above range of $\lambda$ is nonempty for $k$ in the nonempty interval $(1, 4/\sqrt{\pi^2+4\pi-12})$. 
     By \cite[Example 2.1]{QXia2003}, minimizing the stationary energy $\Mcur^\tau$ with boundary conditions $a^\pm[t]$ yields a different path for each fixed $t\in [0, 1]$, even with $\lambda=0$: for example when $t=0$, we have $\rho(t)=1/2$ hence \cite[Example 2.1]{QXia2003} implies the optimal path is a ``Y-shaped'' one where the bottom segment bisects the angle between the two top segments with intersection whose $y$-coordinate is strictly positive.
    \wrapstuffclear
\end{example}

\section{Non-discrete case}\label{sec: nondiscrete} 
We will now extend our notion of transport paths to transport between general probability measure-valued functions.
The energy will be defined as a lower semi-continuous envelope of the discrete energies. To do so, we must specify what notion of approximating sequences to consider, as a first step we define a notion of convergence for the boundary data. 
\begin{defin}\label{def: approximating boundary data}
    If $1<p\leq \infty$, $\mu\in \overstar{W}^{1,p}(\prob_c)$, and $\{\mu_i\}_{i\in \N} \subset \overstar{W}^{1,p}(\prob_c)$, we say \emph{$\mu_i$ $W^{1, p}$-flat converges to $\mu$}, written $\mu_i\Wpflatto \mu$ if
\begin{align*}
    \lim_{i\to\infty}&\left(\Lp{\Fcur(\mu_i- \mu)}+\Lp{\Fcur(\nu_i-\nu)}\right)=0,\\
    \sup_{i\in \N}&\norm{\overstar{\mu}_i'}_{L^{p'}(C_0(\domain))^\ast}<\infty,
\end{align*}
    where $\nu_i$ and $\nu$ correspond to $H$ obtained by applying Remark~\ref{rem: weakly measurable} with  $\overstar{F}=\overstar{\mu}_i'$ and $\overstar{\mu}'$.
\end{defin}
Since we can write the flat norm of a measure as a supremum over a fixed countable family, by Remark~\ref{rem: weakly measurable} we see the $L^p$ norms above are well-defined. 
By Lemma~\ref{lem: Lp weights}, note that if $\mu_i\in W^{1, p}(\atomic)$, then $\nu_i=\overstar{\mu}_i'=\mu'_i$, we will use this fact freely.
\begin{defin}\label{def: approximating seq}
 Suppose we have probability measure-valued functions $\mu^+$ and $\mu^-\in \overstar{W}^{1,p}(\prob_c)$. A vector measure-valued function, $T\in \overstar{W}^{1,p}(\vecmeas)$, is a \emph{$\overstar{W}^{1,p}$ transport path from $\mu^+$ to $\mu^-$} (written $T\in \WStarPath(\mu^+, \mu^-)$) if there exist sequences of atomic measure-valued functions $\{a^\pm_i\}_{i\in \N} \subset W^{1,p}(\atomic)$ and $G_i\in \WDPath(a^+_i, a^-_i)$ such that
 \begin{align*}
     a^\pm_i &\Wpflatto \mu^\pm,\qquad
     G_i[t] \flatto T[t]\text{ for a.e. }t\in [0, 1],\\
     G'_i &\overset{L^{p'}(\covecfield)^\ast}{\wto} \overstar{T}'.
 \end{align*}
We will refer to such a sequence of discrete $W^{1,p}$ transport paths $\{G_i\}$ as a \emph{discrete approximating sequence for T} and will write $(G_i, a_i^+, a_i^-)\approxto (T, \mu^+, \mu^-)$.
\end{defin}
\begin{rem}\label{rem: pointwise implies weak convergence}
If $\mu_i\Wpflatto \mu$, we may pass to a subsequence for which  $\Fcur(\mu_i[\cdot]- \mu[\cdot])$ converges a.e. on $[0, 1]$ to $0$. Since $\Mcur(\mu_i[\cdot])=\Mcur(\mu[\cdot])=1$ for all $i\in \N$ and the boundary of any $0$-current is zero by definition, by Remark~\ref{rem: vector measures and 1 currents} for this subsequence we have $\mu_i[t]\wto \mu[t]$ for a.e. $t\in [0, 1]$.

    Now in general, if $\mu_i[t] \wto \mu[t]$ for a.e. $t\in [0, 1]$, we can show that $\mu_i\overset{L^{p'}(C_0(\domain))^\ast}{\wto} \mu$. Indeed, suppose $\Psi\in L^{p'}(C_0(\domain))$, then for a.e. $t\in [0, 1]$, 
    \begin{align*}
    \lim_{i\to\infty}\int_{\domain} \Psi[t]d\mu_i[t]
    &=\int_{\domain} \Psi[t]d\mu[t],\qquad
        \sup_{i\in \N}\abs{\int_{\domain} \Psi[t]d\mu_i[t]}
        \leq \norm{\Psi[t]}_{C(\domain)}.
    \end{align*}
    Since $\Psi\in L^{p'}(C_0(\domain))$, by dominated convergence we have 
    \begin{align*}
        \lim_{i\to\infty}\int_0^1 \int_{\domain} \Psi[t]d\mu_i[t]dt
        =\int_0^1\int_{\domain} \Psi[t]d\mu[t]dt,
    \end{align*}
     proving the claim.

Additionally, if $\eta\in C^\infty_c((0, 1))$ and $\psi\in C_0(\domain)$, arguing as above using dominated convergence we obtain
\begin{align*}
    \lim_{i\to\infty}\inner{\overstar{\mu}'_i}{\eta \psi} 
    &=-\lim_{i\to\infty}\int_0^1 \eta'(t) \int_{\domain}\psi(y)d\mu_i[t](y)dt\\
    &=-\int_0^1 \eta'(t) \int_{\domain}\psi(y)d\mu[t](y)dt
    =\inner{\overstar{\mu}'}{\eta \psi}.
\end{align*}
Thus if $\sup_{i\in \N}\norm{\overstar{\mu}'_i}_{L^{p'}(C_0(\domain))^\ast}<\infty$ we may use a density argument to see the above limit calculation holds for $\eta \psi$ replaced by an arbitrary element of $L^{p'}(C_0(\domain))$ that is zero outside a compact set in $(0, 1)$, in particular
\begin{align*}
    \overstar{\mu}'_i\overset{L^{p'}(C_0(\domain))^\ast}{\wto} \overstar{\mu}'.
\end{align*}
We comment that it is possible to obtain all of the results in this section, hence Theorem~\ref{thm: minimizers exist}, if we change Definition~\ref{def: approximating boundary data} of $W^{1, p}$-flat convergence to the weaker version 
\begin{align*}
 \lim_{i\to\infty}\Lp{\Fcur(\mu_i-\mu)}=0,\qquad \sup_{i\in \N}\norm{\overstar{\mu}'_i}_{L^{p'}(C_0(\domain))^\ast}<\infty.
\end{align*}
The form given in Definition~\ref{def: approximating boundary data} is only necessary in the proof of the results in Section~\ref{sec: metric}.

     Finally, for any $T_i$, $T\in L^p(\vecmeas)$ with  $\sup_{i\in \N}\esssup_{t\in [0, 1]}\Mcur(T_i[t])<\infty$ and $T_i[t] \flatto T[t]$ for a.e. $t\in [0, 1]$, an argument similar to the one above yields that 
     \begin{align*}
         T_i\overset{L^{p'}(\covecfield)^\ast}{\wto} T.
     \end{align*}
\end{rem}
Now we define which energies we will consider for minimization.
\begin{defin}\label{def: cts system cost function}
Let $1< p \leq \infty$, $\lambda\geq0$,  and $\tau$ be a transportation cost. If $T\in \WStarPath(\mu^+, \mu^-)$ for some $\mu^\pm\in \overstar{W}^{1, p}(\prob_c)$, the energy $\ELp^{\tau, p}_{\lambda}(T)$ is defined by 
\begin{align}\label{eqn: cts system cost}
    \ELp^{\tau, p}_{\lambda}(T) &= \inf \bigg\{ \liminf_{i\to \infty} \DELp^{\tau, p}_{\lambda}(G_i)
        \mid (G_i, a_i^+, a_i^-)\approxto (T, \mu^+, \mu^-) \bigg\}.
\end{align}
\end{defin}
We quickly verify the basic but important fact that these new energies match those defined previously for discrete $W^{1, p}$ paths that are never cyclic.
\begin{prop}\label{prop: energies match}
    If $\tau$ is a transportation cost, $1< p\leq\infty$, and $G$ is a discrete $W^{1, p}$ path that is never cyclic, then 
    \begin{align*}
        \ELp^{\tau, p}_{\lambda}(G)=
        \DELp^{\tau, p}_{\lambda}(G).
    \end{align*}
    \end{prop}
\begin{proof}
    Taking the constant sequence equal to $G$, it is immediate that $\ELp^{\tau, p}_{\lambda}(G)\leq
        \DELp^{\tau, p}_{\lambda}(G)$. 
        To obtain the opposite inequality, suppose $G\in \WDPath(a^+, a^-)$ and $(G_i, a^+_i, a^-_i)\approxto (G, a^+, a^-)$, 
        we may also assume that $\liminf_{i\to\infty}\DELp^{\tau, p}_{\lambda}(G_i)<\infty$. 
         For a.e. $t\in [0, 1]$, we have $G_i[t]\flatto G[t]$, 
         and clearly each $G_i[t]$ and $G[t]$ are rectifiable $1$-currents. 
        If we define $H(\theta)=\tau(\abs{\theta})$, this  satisfies \cite[Assumption 2.1]{ColomboDeRosaMarcheseStuvard17}, hence we can apply \cite[Proposition 2.6]{ColomboDeRosaMarcheseStuvard17} to the $\Mcur_H$-masses of $G_i[t]$ for a.e. $t\in [0, 1]$ then use Fatou's lemma (or just by first principles when $p=\infty$) to obtain
        \begin{align*}
            \MLp^\tau_{p}(G)\leq \liminf_{i\to\infty}\MLp^\tau_{p}(G_i).
        \end{align*}
        Since $G$ is never cyclic, we have Lemma~\ref{lem: eliminating cycles}~\eqref{eqn: never cyclic energy}. Then as $G'_i$ weak-$\ast$ converges to $G'$ by Remark~\ref{rem: pointwise implies weak convergence}, using the weak-$\ast$ lower semi-continuity of the $L^p$ norm, we obtain  
        \begin{align*}
            \DELp^{\tau, p}_{\lambda}(G)
            \leq \liminf_{i\to\infty}\DELp^{\tau, p}_{\lambda}(G_i).
        \end{align*}
        By taking an infimum over all discrete approximating sequences we obtain the inequality  $\ELp^{\tau, p}_{\lambda}(G)\geq
        \DELp^{\tau, p}_{\lambda}(G)$.
\end{proof}

\subsection{Technical choices}\label{subsec: model choice}

A possible starting point is to consider an energy similar to \cite{BrancoliniWirth18} for paths belonging to $L^p(\vecmeas)$ instead of $\overstar{W}^{1, p}(\vecmeas)$. However, this quickly leads to difficulties in the main existence theorem, as the only source of compactness available would be weak-$\ast$ compactness via the Banach-Alaoglu theorem of a minimizing sequence in $L^{p'}(C_0(\domain; \Lambda^1(\R^n)))^\ast$; as mentioned in the introduction, this space is strictly larger than $L^p(\vecmeas)$, and it is desirable for minimizers to belong to the latter. The approach we adopt in this paper is to convert time-dependent paths into $2$-currents, and attempt to obtain compactness after slicing in the time variable. In order to be able to obtain sufficient compactness with this approach, one needs some reasonable control of the paths in the time variable, which leads to considering paths in $W^{1, p}(\vecmeas)$, with an added derivative term in the energy. However, this again leads to issues as this time the weak derivative will be subject to the same issue of weak-$\ast$ compactness being insufficient to obtain a minimizer in the same class. The key observation is that $\overstar{W}^{1, p}(\vecmeas)$ remains closed even with only weak-$\ast$ convergence of the weak-$\ast$ derivatives, and that for \emph{discrete} transport paths, belonging to $\overstar{W}^{1, p}(\vecmeas)$ is equivalent to the \emph{a priori} stronger notion of membership in $W^{1, p}(\vecmeas)$ when $p>1$ (i.e., Lemma~\ref{lem: Lp weights}).

However, this leads to another difficulty regarding the presence of cycles. A key step in the time-independent case of \cite{BrancoliniWirth18} is showing that one can reduce the associated energy by removing strong cycles. In the time-dependent case, since we want some quantity related to the derivative of a transport path to appear in the energy, a naive choice will not be amenable to such a cycle removal procedure.  For example, if we had decided to take the energy as $\MLp^\tau_{p}(G)+\lambda \MLp^1_{p}(G')$, ``subtracting off a cycle'' will reduce the first term, but the second term with $G'$ may end up increasing. 

In addition to Lemma~\ref{lem: Lp weights}, the restriction $p>1$ in the Sobolev regularity also plays a subtle but crucial role in obtaining strong measurability of the minimizer after extracting a convergent subsequence. Specifically, we require $p>1$, as when $p=1$ we cannot obtain uniform H\"older estimates as in Lemma~\ref{lem: bounded discrete uniformly Holder}~\eqref{eqn: uniform Holder}; the case of $p=1$ is left for a future work. The H\"older regularity (and consequently the choice of equipping $\vecmeas$ with the total variation norm) also plays a part in showing existence of a time-independent countably $1$-rectifiable set which supports a minimizer. 

As mentioned, from a modeling viewpoint one should not alter the underlying physical pathways to accommodate changes in source and target distribution, but one would still want the transport to be supported on a fundamentally one-dimensional object. The control of the derivative that the Sobolev regularity offers also makes sense from a physical perspective. If we are using these paths to model blood flow in the human body, for example, weights whose derivatives have large norms would correspond to rapid changes in heartrate or blood pressure. This would, in fact, imply disorders such as labile hypertension (rapid changes in blood pressure) and Postural Orthostatic Tachycardia Syndrome (POTS) (rapid changes in heart rate caused by a change in posture), that can be dangerous and are outside the norm. Similarly, if we are using these paths to model traffic on the highway, weights with large derivative model abrupt shifts between light traffic and very heavy traffic, which is a likely marker of congestion or accidents. Finally, for systems such as blood flow and traffic, it is natural to prohibit the formation of cycles. There are other systems in which cycles may naturally be observed, and a framework to accommodate such settings is left for future work.

\subsection{Existence of minimizers}
Next, we work towards a compactness result crucial for existence of minimizers. We start with a uniform H\"older continuity result.

\begin{lem}\label{lem: bounded discrete uniformly Holder}
    Suppose $1< p\leq\infty$ and $\{G_i\}_{i\in \N}$ is a sequence of discrete $W^{1, p}$ transport paths such that
    \begin{align*}
        C_p:=\sup_{i\in \N}\norm{G'_i}_{L^p(\vecmeas)}<\infty.
    \end{align*}
    \begin{enumerate}[(i)]
            \item \label{eqn: mass Lp converges} If   $\{\norm{G_i}_{L^p(\vecmeas)}\}_{i\in \N}$ is uniformly bounded,
        there exists a subsequence such that the sequence of real-valued functions $\{\TV{G_i}\}_{i\in \N}$ on $[0, 1]$ converges in $L^p$.
        \item \label{eqn: uniform Holder} Each $G_i$ is $(1-1/p)$-H\"older continuous as a map from $[0, 1]\to (\vecmeas, \TV{\cdot})$, with H\"older seminorm uniformly bounded by $C_p$, with the convention $1/\infty=0$.
    \end{enumerate}
    
\end{lem}
\begin{proof}
Suppose $G_i=\sum_{e\in E(G_i)}w_i(e, \cdot)[[e]]$. 
By the hypothesis of \eqref{eqn: mass Lp converges}, the sequence of real-valued functions
\begin{align*}
    t\mapsto \sum_{e\in E(G_i)}\abs{w_i(e, t)}\HOne{e}=\TV{G_i[t]}
\end{align*}
are uniformly bounded in $W^{1, p}$, hence by the Rellich-Kondrachov theorem (see \cite[Theorem 8.8 (7)]{Brezis11book}) we obtain \eqref{eqn: mass Lp converges}.

Now for $0\leq t_1<t_2\leq 1$, and any $\xi\in \covecfield$, we have
\begin{align*}
    \abs{\inner{G_i[t_1]-G_i[t_2]}{\xi}}
    &\leq\sum_{e\in E(G_i)}\abs{w_i(e, t_1)-w_i(e, t_2)}\abs{\inner{[[e]]}{\xi}}\\
    &\leq \sum_{e\in E(G_i)}\left(\int_{t_1}^{t_2}\abs{w'_i(e, s)}ds\right)\HOne{e}\norm{\xi}_{\covecfield}\\
    &\leq \norm{G'_i}_{L^p(\vecmeas)}\abs{t_2-t_1}^{{1-\frac{1}{p}}}\norm{\xi}_{\covecfield},
\end{align*}
taking a supremum over $\xi\in \covecfield$ with unit norm yields
\begin{align*}
    \TV{G_i[t_1]-G_i[t_2]}\leq C_p\abs{t_2-t_1}^{{1-\frac{1}{p}}}
\end{align*}
proving \eqref{eqn: uniform Holder}.
\end{proof}
Next we prove a key boundedness lemma for minimizing sequences. Carefully note, we do \emph{not} claim $G_i\approxto T$ below.
\begin{lem}\label{lem: mass bounded}
Suppose $1< p\leq \infty$, $\lambda>0$, and $\tau$ is a transportation cost.
    \begin{enumerate}[(i)]
        \item\label{eqn: bounded approx seq} If $T\in \WStarPath(\mu^+, \mu^-)$ for some $\mu^\pm\in \overstar{W}^{1,p}(\prob_c)$ with $\ELp^{\tau, p}_{\lambda}(T)<\infty$, then there exist $\{a^\pm_i\}_{i\in \N} \subset W^{1,p}(\atomic)$ and never cyclic $G_i\in \WDPath(a^+_i, a^-_i)$ such that
        \begin{align*}
        \lim_{i\to\infty}\DELp^{\tau, p}_{\lambda}(G_i)&\leq\ELp^{\tau, p}_{\lambda}(T),\qquad
            a^\pm_i \Wpflatto \mu^\pm,\qquad
            \sup_{i\in \N}\norm{G_i}_{L^{p}(\vecmeas)}<\infty.
        \end{align*}
        \item\label{eqn: bounded seq for distance} If $\mu^\pm\in \overstar{W}^{1, p}(\prob_c)$, and 
        $\dist^{\tau, p}_{\lambda}(\mu^+, \mu^-)<\infty$, there exist $\{a^\pm_i\}_{i\in \N} \subset W^{1,p}(\atomic)$ and never cyclic  $G_i\in \WDPath(a^+_i, a^-_i)$ such that
        \begin{align*}
            \dist^{\tau, p}_{\lambda}(\mu^+, \mu^-)
            &=\lim_{i\to\infty}\DELp^{\tau, p}_{\lambda}(G_i),\\
            a^\pm_i &\Wpflatto \mu^\pm,\qquad
            \sup_{i\in \N}\norm{G_i}_{L^{p}(\vecmeas)}
            <\infty.
        \end{align*}
        \item\label{eqn: same compact support}
        In case~\eqref{eqn: bounded seq for distance} above, additionally if there is a compact, convex set $K\subset\domain$ with nonempty interior such that $\mu^\pm\in \overstar{W}^{1, p}(\prob(\compact))$, then all $G_i$ can be taken to satisfy $\bigcup_{e\in E(G_i)}e\subset K$.
    \end{enumerate}
\end{lem}
\begin{proof}
    Suppose $T$ is a $\overstar{W}^{1, p}$ transport path with $\ELp^{\tau, p}_{\lambda}(T)<\infty$, then there exists at least one discrete approximating sequence $\{\hat{G}_i\}_{i\in \N}$ for $T$ with
    \begin{align}
        \ELp^{\tau, p}_{\lambda}(T)
        &=\lim_{i\to\infty}\DELp^{\tau, p}_{\lambda}(\hat{G}_i),\qquad
        \sup_{i\in \N}\DELp^{\tau, p}_{\lambda}(\hat{G}_i)
        <\infty,\label{eqn: initial sequence}
    \end{align}
    let us assume that $\hat{G}_i\in \WDPath(a^+_i, a^-_i)$ for some $a^\pm_i\in W^{1, p}(\atomic)$, then $a^\pm_i\Wpflatto \mu^\pm$. Without loss, we may also assume that $\spt a^+_i\cap \spt a^-_i=\emptyset$: if there is any point $x\in \spt a^+_i\cap \spt a^-_i$, one can modify $\hat{G}_i$ by adding an arbitrarily short edge directed from $x$, with weight $a^-_i[t](\{x\})$ in a way that the terminal point $x'$ is not contained in $\spt a^+_i\cup \spt a^-_i$, then replace the term $a^-_i[t](\{x\})\delta_x$ by $a^-_i[t](\{x\})\delta_{x'}$ in $a^-$. If the additional edges are of sufficiently short length, this can be done without changing~\eqref{eqn: initial sequence} and that the sequence of paths is still a discrete approximating sequence. Then for each $i\in \N$, apply Lemma~\ref{lem: eliminating cycles} to find a $G_i\in \WDPath(a^+_i, a^-_i)$ that is never cyclic, and 
            $\DELp^{\tau, p}_\lambda(G_i)
            \leq \DELp^{\tau, p}_\lambda(\hat{G}_i)$,
in particular this sequence satisfies
 \begin{align*}
        \lim_{i\to\infty}\DELp^{\tau, p}_{\lambda}(G_i)
        &\leq\ELp^{\tau, p}_{\lambda}(T).
    \end{align*}
    By combining \cite[Lemma 2.9]{BrancoliniWirth18} and Lemma~\ref{lem: subadditive bound}, we find that 
    \begin{align*}
        \norm{G_i}_{L^{p}(\vecmeas)}&=\Lp{\sum_{e\in E(G)}w_i(e, \cdot)\HOne{e}}\\
        &\leq\rho(\tau, 1)^{-1}\Lp{\sum_{e\in E(G)}\tau(\abs{w_i(e, \cdot)})\HOne{e}}
        \leq \rho(\tau, 1)^{-1} \sup_{i'\in \N}\DELp^{\tau, p}_{\lambda}(G_{i'})<\infty,
    \end{align*}
    hence
    \begin{align*}
        \sup_{i\in \N}\norm{G_i}_{L^{p}(\vecmeas)}&<\infty.
    \end{align*}
   This finishes the proof of claim \eqref{eqn: bounded approx seq}. Claim \eqref{eqn: bounded seq for distance} follows by taking discrete approximating sequences for a  minimizing sequence of $\dist^{\tau, p}_\lambda$, and then taking a diagonal sequence and applying claim \eqref{eqn: bounded approx seq}; to obtain $\sup_{i\in \N}\norm{(a^\pm_i)'}_{L^{p'}(C_0(\domain))^\ast}<\infty$, we can recall Remark~\ref{rem: vector measures and 1 currents} and pass to an appropriate subsequence which weak-$\ast$ converges.

        Now suppose $K$ is as in the hypothesis of claim~\eqref{eqn: same compact support}. Let $\pi_K: \R^n\to K$ be the closest point projection, i.e. $\pi_K(x)=z$ where $\abs{z-x}=\min_{z'\in K}\abs{z'-x}$; by convexity and compactness of $K$, $\pi_K$ is well-defined and contractive (see \cite[Theorem 1.2.1]{Schneider14book}). 
        Then if $\{\hat{G}_i\}_{i\in \N}$ is the diagonal sequence obtained from claim~\eqref{eqn: bounded seq for distance}, for each $i\in \N$ we can define a map $\Xi_i: V(\hat{G}_i)\to K$ by $\Xi^K_i(y):=\pi_K(y)+v_i(y)$, choosing $v_i(y)$ so that the edges $\overrightarrow{\Xi^K_i(e^+), \Xi^K_i(e^-)}$ for $e\in E(\hat G_i)$ intersect in at most one point, except for edges that are reversals of eachother, and 
        \begin{align*}
            \sum_{e\in E(\widetilde{G}_i)}\left(\Lp{\hat{w}_i(e, \cdot)}+\Lp{\hat{w}'_i(e, \cdot)}\right)\left(\abs{v_i(e^+)}+\abs{v_i(e^-)}\right)<2^{-i}.
        \end{align*}
        Then define the graph $\widetilde{G}_i$ where 
        \begin{align*}
        V(\widetilde{G}_i)
        :&=\{\Xi^K_i(y)\mid y\in V(\hat G_i)\},\qquad
        E(\widetilde{G}_i)
        :=\{\overrightarrow{\Xi^K_i(e^+), \Xi^K_i(e^-)}\mid e\in E(\hat G_i)\}
        \end{align*}
        and the weight function for $\overrightarrow{\Xi^K_i(e^+), \Xi^K_i(e^-)}$ is the weight function $\hat{w}_i(e, \cdot)$ for $e$ in $\hat G_i$.  Since $\pi_K$ is contractive,
            $\HOne{\overrightarrow{\Xi^K_i(e^+), \Xi^K_i(e^-)}}\leq \HOne{e}+\abs{v_i(e^+)}+\abs{v_i(e^-)}$,
        thus 
        \begin{align}
            \norm{\widetilde{G}_i}_{L^{p}(\vecmeas)}&\leq \norm{\hat G_i}_{L^{p}(\vecmeas)}+2^{-i},\qquad
            \norm{\widetilde{G}'_i}_{L^p(\vecmeas)}\leq \norm{\hat G'_i}_{L^p(\vecmeas)}+2^{-i}.\label{eqn: energy reduction projection}
        \end{align}
      Moreover, $\widetilde{G}_i\in \WDPath((\Xi^K_i)_\sharp a^+_i, (\Xi^K_i)_\sharp a^-)$, so we can apply the same procedure as claim~\eqref{eqn: bounded seq for distance} to obtain never cyclic discrete transport paths whose edges remain in $\compact$.
\end{proof}

Recall that if $\tilde T\in \mathcal{D}_m(D)$, $\pi: D\to \R^l$ for some $l\leq m$, and $t\in \pi(D)$, then $\langle \tilde T, \pi, t\rangle\in \mathcal{D}_{m-l}(D)$ is the \emph{slice of $\tilde T$ in $\pi^{-1}(\{t\})$}; see \cite[4.3.1]{Federer96}, we will make use of various properties of slices but will not need the exact definition. The following is the central compactness result of this paper.

\begin{prop}\label{prop: compactness result}
    Suppose $1< p\leq \infty$, $\{a^\pm_i\}_{i\in \N}\subset \overstar{W}^{1, p}(\atomic)$, $\{\mu^\pm\}\subset \overstar{W}^{1, p}(\prob_c)$, and $G_i\in \WDPath(a^+_i, a^-_i)$ for each $i\in \N$ are such that for some compact $\compact\Subset \domain$,
    \begin{align*}
        \sup_{i\in \N}&(\norm{G_i}_{L^{p}(\vecmeas)}+\norm{G'_i}_{L^p(\vecmeas)})<\infty,\qquad
        \bigcup_{i\in\N}\bigcup_{e\in E(G_i)}e\subset \compact,\qquad
        a^\pm_i \Wpflatto \mu^\pm.
    \end{align*}
    Then there exist $T\in \WStarPath(\mu^+, \mu^-)$ and a subsequence such that $(G_i, a^+_i, a^-_i)\approxto (T, \mu^+, \mu^-)$.
    
Additionally we may assume $T$ is a $(1-1/p)$-H\"older continuous function from $[0, 1]$ to $(\vecmeas, \TV{\cdot})$ and $G_i[t]\flatto T[t]$ for \emph{every} $t\in [0, 1]$.
\end{prop}
\begin{proof}
    Note that $\twocurrent(G_i)$ are all supported on the compact set $[-1, 2]\times  \compact\Subset \R \times \domain$ 
    A quick calculation yields $\Mcur(\twocurrent(G_i))=\MLp^1_1(G_i)$, combining with Remark~\ref{rem: boundary of transport path} we find
    \begin{align*}
        \sup_{i\in \N}(\Mcur(\twocurrent(G_i))+\Mcur(\partial \twocurrent(G_i)))<\infty.
    \end{align*}
    Thus by \cite[4.2.17 Compactness Theorem]{Federer96}, we can pass to a subsequence to assume there exists a  $\tilde{T}\in \mathcal{D}_2(\R\times \domain)$ supported in $[-1, 2]\times \compact$
    such that 
    \begin{align}\label{eqn: limit is normal current}
        \twocurrent(G_i)\flatto\tilde{T},\qquad \Mcur(\tilde{T})+\Mcur(\partial \tilde{T})<\infty.
    \end{align}
     Let $\pi: [-1, 2]\times \domain\to [-1, 2]$ be the projection $\pi(t, x)=t$, then we can apply \cite[4.3.1 p.437]{Federer96} to obtain
\begin{align*}
    0&\leq\int_{-1}^{2}\Fcur(\langle \twocurrent(G_i)-\tilde{T}, \pi, t\rangle)dt\leq [\pi]_{\Lip([0, 1]\times \overline{\domain})}\Fcur(\twocurrent(G_i)-\tilde{T})\to 0,\ i\to\infty.
\end{align*}
Passing to another subsequence, for a.e. $t\in [-1, 2]$ we have $\Fcur(\langle \twocurrent(G_i)-\tilde{T}, \pi, t\rangle)\to 0$, in particular (recall Remark~\ref{rem: vector measures and 1 currents}) $\langle \twocurrent(G_i), \pi, t\rangle\currentto\langle\tilde{T}, \pi, t\rangle$.

Now suppose $\phi=\sum_{j=1}^n \phi_j dx^j\in \mathcal{D}^1(\domain)$, and fix $t_0\in (0, 1)$ and $0<h<
\min(t_0, 1-t_0)$. Then for any $\eta\in C^\infty_c((0, 1))$ with $0\leq \eta\leq 1$ and $\eta\equiv 1$ on $[t_0-h, t_0+h]$, applying \cite[4.3.2 Theorem (1)]{Federer96} with $\Phi=\mathds{1}_{[t_0-h, t_0+h]}$ and $\psi(t, x)=\eta(t)\phi(x)$ there (also using that $\langle \twocurrent(G_i), \pi, t\rangle$ is supported on $\{t\}\times \compact$ by \cite[4.3.1 p.437]{Federer96}) yields that
\begin{align*}
    \int_{t_0-h}^{t_0+h}\inner{\langle \twocurrent(G_i), \pi, t\rangle }{\phi}dt
    &=\int_{t_0-h}^{t_0+h}\inner{G_i[t]}{\sum_{j=1}^n \phi_j dx^j}dt.
\end{align*}
Letting $\mathfrak{D}$ be a dense, countable subset of $\mathcal{D}^1(\domain)$, by the Lebesgue differentiation theorem there exists a Lebesgue null set $\mathcal{N}\subset [0, 1]$ such that $\inner{\langle \twocurrent(G_i), \pi, t\rangle }{\phi}=\inner{G_i[t]}{\phi}$ for all $t\in [0, 1]\setminus \mathcal{N}$, $i\in \N$, and any $\phi\in \mathfrak{D}$. Taking a limit, we obtain for any $\phi\in \mathfrak{D}$,
\begin{align}\label{eqn: limit of slices}
    \inner{\langle \tilde{T}, \pi, t\rangle}{\phi}=\lim_{i\to\infty}\inner{G_i[t]}{\phi}.
\end{align}
By Lemma~\ref{lem: bounded discrete uniformly Holder}~\eqref{eqn: mass Lp converges}, there is a subsequence such that the sequence $\{\TV{G_i}\}_{i\in \N}=\{\Mcur(G_i)\}_{i\in \N}$ converges in $L^p$ on $[0, 1]$, passing to another subsequence we may also assume the sequence converges pointwise a.e. on $[0, 1]$. In particular, 
\begin{align}\label{eqn: a.e. mass bounded}
    \sup_{i\in \N}\Mcur(G_i[t])<\infty, \qquad \text{a.e. }t\in [0, 1],
\end{align}
 thus \eqref{eqn: limit of slices} holds for \emph{all} $\phi\in \mathcal{D}^1(\domain)$. Since for such $t$,
\begin{align*}
    \sup_{i\in \N}(\Mcur(G_i[t])+\Mcur(\partial G_i[t]))
    &=\sup_{i\in \N}(\Mcur(G_i[t])+\Mcur( a^-_i[t]-a^+_i[t]))
    \leq \sup_{i\in \N}(\Mcur(G_i[t])+2),
\end{align*}
recalling Remark~\ref{rem: vector measures and 1 currents} we have $G_i[t]\flatto \langle \tilde{T}, \pi, t\rangle$ for a.e. $t\in [0, 1]$. By the lower semi-continuity of $\Mcur$ with respect to flat norm convergence, we see $\Mcur(\langle \tilde{T}, \pi, t\rangle)<\infty$, hence $\langle \tilde{T}, \pi, t\rangle\in \vecmeas$.

Let us define $T[t]:=\langle \tilde{T}, \pi, t\rangle$ for a.e. $t\in [0, 1]$; by the Banach-Alaoglu theorem we may pass to one final subsequence to obtain some $\hat{T}\in L^{p'}(\covecfield)^\ast$ such that $G'_i\overset{L^{p'}(\covecfield)^\ast}{\wto} \hat{T}$. Since $p>1$ the uniform H\"older continuity of the $G_i$ from Lemma~\ref{lem: bounded discrete uniformly Holder}~\eqref{eqn: uniform Holder} and the weak-$\ast$ lower semi-continuity of the total variation norm, along with completeness of $(\vecmeas, \TV{\cdot})$, we see that $T$ can be extended to a $(1-1/p)$-H\"older continuous function from $[0, 1]$ to $(\vecmeas, \TV{\cdot})$. This implies that the image $T([0, 1])$ is a separable subset of $(\vecmeas, \TV{\cdot})$, and $t\mapsto \inner{T[t]}{\xi}$ is a Borel function for any $\xi\in \covecfield$. Thus by \cite[Chapter II.1, Corollary 4]{DiestelUhl77}, $T$ is strongly measurable. Additionally, the uniform H\"older continuity of the $G_i$ implies the flat convergence $G_i[t]\flatto T[t]$ can be extended to \emph{all} $t\in [0, 1]$. By the weak-$\ast$ lower semi-continuity of $\TV{\cdot}$, and Fatou's lemma when $p<\infty$ and by definition when $p=\infty$,
\begin{align*}
    \norm{T}_{L^{p}(\vecmeas)}
    &=\Lp{\TV{T[\cdot]}}
    \leq \Lp{\liminf_{i\to\infty}\TV{G_i[\cdot]}} 
    \leq \liminf_{i\to\infty}\Lp{\TV{G_i[\cdot]}}\\
    &\leq \sup_{i\in \N}\Lp{\TV{G_i[\cdot]}}
    =\sup_{i\in \N}\MLp^1_{p}(G_i)<\infty,
\end{align*}
thus $T\in L^{p}(\vecmeas)$. Next, we note that by \eqref{eqn: a.e. mass bounded} combined with Lemma~\ref{lem: bounded discrete uniformly Holder}~\eqref{eqn: uniform Holder} we have \\$\sup_{i\in \N}\sup_{t\in [0, 1]}\TV{G_i[t]}<\infty$, hence by Remark~\ref{rem: pointwise implies weak convergence} we have $G_i \overset{L^{p'}(\covecfield)^\ast}{\wto} T$. Then for $\eta\in C^\infty_c((0, 1))$, $\xi\in \covecfield$, we calculate
\begin{align*}
        -\inner{\hat{T}}{\eta \xi}
        &=-\lim_{i\to\infty}\inner{G'_i}{\eta \xi}
        =\lim_{i\to\infty}\int_0^1 \inner{G_i[t]}{\eta'(t)\xi}dt
        =\int_0^1\eta'(t)\inner{T[t]}{\xi}dt,
    \end{align*}
   thus $T\in \overstar{W}^{1, p}(\vecmeas)$ with $\hat{T}=\overstar{T}'$, finishing the proof.
\end{proof}
With this compactness result in hand, by an argument reminiscent of \cite[Proposition 3.1]{QXia2003} and \cite[Proposition 2.18]{BrancoliniWirth18}, we can show the existence of $W^{1, p}$ transport paths with finite $\ELp^{\tau, p}_{\lambda}$ energy. Following \cite[Definition 2.16]{BrancoliniWirth18}, we define some notation.
\begin{defin}\label{def: n-adic mass fluxes}
     Let $\mu\in \overstar{W}^{1, p}(\prob)$.
    \begin{enumerate}
        \item\label{eqn: elementary flux} For $s\in (0, \infty)$ and $x\in \domain$, let $\{Q^{x, s}_j\}_{j=1}^{2^n}$ be the collection of hypercubes of the form $x+v+[-s, s)^n$ where $v\in \{-s, s\}^n$, with centers $\{c^{x, s}_j\}_{j=1}^{2^n}$. Then define the discrete $W^{1, p}$ path $G_{\mu, s, x}$ by 
        \begin{align*}
            V(G_{\mu, s, x}):&=\{x\}\cup \{c^{x, s}_j\}_{j=1}^{2^n},\\
            E(G_{\mu, s, x}):&=\{\overrightarrow{x, c^{x, s}_j}\mid j=1,\ldots 2^n\},\\
            w(\overrightarrow{x, c^{x, s}_j}, t):&=\mu[t](Q^{x, s}_j),\qquad  j=1,\ldots, 2^n.
        \end{align*}
        \item For $k\in \N$, $s\in (0, \infty)$, and $x\in \domain$, inductively define the discrete $W^{1, p}$ paths $G^k_{\mu, s, x}$ by
        \begin{align*}
            G^1_{\mu, s, x}:&=G_{\mu, s, x},\qquad
            G^k_{\mu, s, x}:=G_{\mu, s, x}+ \sum_{v\in \{-s, s\}^n}G^{k-1}_{\mu, s/2, x+v},
        \end{align*}
        We let $G^k_\mu:=G^k_{\mu, 1, 0}$, and define
        \begin{align*}
            F^k_\mu:&=
            \begin{cases}
            G^1_\mu,&k=1,\\
                G^k_\mu- G^{k-1}_\mu,&k>1.
            \end{cases}
        \end{align*}
        \item For $k\in \N$, define 
        \begin{align*}
        L^0:&=\{0\},\qquad
        L^k:=\{c^{x, 2^{1-k}}_j\mid x\in L^{k-1},\ j=1, \ldots, 2^n\},\qquad k\geq 1,\\
            P^k(\mu):&=\sum_{x\in L^{k-1}}\sum_{j=1}^{2^n}\mu[\cdot](Q^{x, 2^{1-k}}_j)\delta_{c^{x, 2^{1-k}}_j},\qquad k\geq 1.
        \end{align*}
    \end{enumerate}
\end{defin}
\begin{lem}\label{lem: Pk approximates}
    If $\mu\in \overstar{W}^{1, p}(\prob([-1, 1)^n))$ for $1<p\leq \infty$, we have $P^k(\mu)\in W^{1, p}(\atomic)$ and $P^k(\mu)\Wpflatto \mu$ as $k\to\infty$.
\end{lem}
\begin{proof}
 Suppose that $\psi\in C^\infty_c(\domain)$ with $\max(\norm{\psi}_{C(\domain)}, \norm{d\psi}_{C(\domain)})\leq 1$. Then for any $k\in \N$ and $t\in [0,1]$, we see
 \begin{align}
     \int_{\domain}\psi d(P^k(\mu)[t]-\mu[t])
     &=\sum_{x\in L^{k-1}}\sum_{j=1}^{2^n}\int_{Q^{x, 2^{1-k}}_j}\left(\psi(c^{x, 2^{1-k}}_j)-\psi\right)d\mu[t]\notag\\
     &\leq \sqrt{n}2^{1-k}\sum_{x\in L^{k-1}}\sum_{j=1}^{2^n}\mu[t]\left(Q^{x, 2^{1-k}}_j\right)
     =\sqrt{n}2^{1-k},\label{eqn: Lid estimate}
 \end{align}
  hence taking a supremum over such $\psi$ shows that for any $1<p\leq\infty$
 \begin{align*}
     \lim_{k\to\infty} \Lp{\Fcur(P^k(\mu)-\mu)}=0.
 \end{align*}
 
Now let $\nu$ be the map corresponding to $H$ obtained from applying Remark~\ref{rem: weakly measurable} with $\overstar{F}=\overstar{\mu}'$. Let $Q=\prod_{i=1}^n [a_i, b_i)$ be any hypercube. Then for $\ell\in \N$ we can apply Urysohn's lemma to find functions $\{\psi_{Q, \ell}\}_{\ell\in \N}\subset C_0(\domain)$ valued in $[0, 1]$, identically $1$ on $\prod_{i=1}^n [a_i, b_i-\ell^{-1}]$ and identically $0$ on $\domain\setminus \prod_{i=1}^n (a_i-\ell^{-1}, b_i)$; note this sequence converges pointwise to $\mathds{1}_{Q}$. For any $\eta\in C^\infty_c((0, 1))$ we have by dominated convergence (as $\abs{\eta'(t)\int_{\domain} \psi_{Q, \ell}d\mu[t]}\leq \abs{\eta'(t)}$ which is integrable on $\R$),
\begin{align*}
    \int_0^1\eta'(t)\mu[t](Q)dt
    &=\int_0^1\eta'(t)\lim_{\ell\to\infty}\int_{\domain} \psi_{Q, \ell}d\mu[t]dt\\
    &=\lim_{\ell\to\infty}\int_0^1\eta'(t)\int_{\domain} \psi_{Q, \ell}d\mu[t]dt\\
    &=-\lim_{\ell\to\infty}\inner{\overstar{\mu}'}{\eta \psi_{Q, \ell}}
    =-\lim_{\ell\to\infty}\int_0^1\eta(t)\inner{\nu[t]}{ \psi_{Q, \ell}}dt.
\end{align*}
Since $\nu[t]$ has finite total variation norm for a fixed $t$, we see by dominated convergence (for signed measures) that for a.e. $t\in [0, 1]$,
\begin{align*}
    \lim_{l\to\infty}\inner{\nu[t]}{ \psi_{Q, l}}
    =\nu[t](Q),
\end{align*}
and as a pointwise a.e. limit of a sequence of Borel functions, $t\mapsto \nu[t](Q)$ is also Borel. 
Then since
\begin{align*}
    \abs{\eta(t)\inner{\nu[t]}{ \psi_{Q, \ell}}}
    &\leq \abs{\eta(t)}\TV{\nu[t]}\norm{\psi_{Q, \ell}}_{C(\domain)}
    \leq \norm{\eta}_{C((0, 1))}\TV{\nu[t]}
\end{align*}
which is $\mathcal{L}^1$-integrable on $(0, 1)$ by \eqref{eqn: weak measurable rep norm} and the compact support of $\eta$, we can again use dominated convergence to find
\begin{align}\label{eqn: weak derivative weights}
    \int_0^1\eta'(t)\mu[t](Q)dt
    &=-\int_0^1\eta(t)\nu[t](Q)dt.
\end{align}
Then we calculate
\begin{align*}
    \norm{\nu[\cdot](Q)}_{L^p}
    &
    \leq \norm{\liminf_{\ell\to\infty}\left(\TV{\nu[\cdot]}\norm{\psi_{Q, \ell}}_{C(\domain)}\right)}_{L^p}
    \leq \norm{\TV{\nu[\cdot]}}_{L^p}<\infty
\end{align*}
by~\eqref{eqn: weak measurable rep norm}, hence $\nu[\cdot](Q)\in L^p$. By~\eqref{eqn: weak derivative weights} we see $\nu[\cdot](Q)$ is the weak derivative of $\mu[\cdot](Q)$, thus we find $\mu[t](Q)\in W^{1,p}$ for any hypercube $Q$.

Now let us write $a_k:=P^k(\mu)$. Let $c \in \spt a_k$, and let $Q_c$ be the hypercube of the form $c+[-2^{2-k},2^{2-k})^n$ where $a_k[t](\{c\})=\mu[t](Q_c)$. Then by the above calculation and Lemma~\ref{lem: Lp weights} we see $a_k\in W^{1, p}(\atomic)$ with 
\begin{align*}
    a'_k=\sum_{c\in \spt a_k}\nu[\cdot](Q_c)\delta_c,
\end{align*}
in particular for all $k\in \N$,
\begin{align*}
    \Lp{\TV{a'_k}}\leq \Lp{\TV{\nu}}.
\end{align*}
The same calculation as in~\eqref{eqn: Lid estimate} yields $\Fcur(a'_k[t]-\nu[t])\leq 2^{1-k}\TV{\nu[t]}$, thus
\begin{align*}
     \lim_{k\to\infty} \Lp{\Fcur(a'_k-\nu)}=0,
 \end{align*}
finishing the proof.
\end{proof}
We are now able to show existence of finite energy transport paths.
\begin{prop}\label{prop: scaling paths}
    Suppose $p>1$ and $\tau$ is an admissible transportation cost. For any $\mu\in \overstar{W}^{1,p}(\prob([-1,1)^n))$, and $k$, $\ell\in \N$ with $k<\ell$ there exists a never cyclic path\\ $G^{k, \ell}\in \WDPath(P^k(\mu), P^\ell(\mu))$ with 
    \begin{align*}
        \MLp^\tau_p(G^{k, \ell}_\mu)
        &\leq \sqrt{n}\sum_{j=k}^{\ell-1} 2^{j(n-1)}\beta(2^{-jn}),\qquad
        \norm{(G^{k, \ell}_\mu)'}_{L^p(\vecmeas)}
        \leq \sqrt{n}\norm{\overstar{\mu}'}_{L^{p'}(\meas)^\ast}\sum_{j=k}^{\ell-1} 2^{-j}.
    \end{align*}
\end{prop}
\begin{proof}
Define 
\begin{align*}
    G^{k, \ell}_\mu
    :=\sum_{j=k}^\ell F^j_\mu,
\end{align*}
then clearly $G^{k, \ell}$ is never cyclic and belongs to $\WDPath(P^k(\mu), P^\ell(\mu))$.

Now note that
\begin{align*}
    G^{k, \ell}_\mu[t]=\sum_{j=k}^{\ell-1} \sum_{x\in L^{j}}\sum_{i=1}^{2^n}\mu[t](Q^{x, 2^{-j}}_i)[[\overrightarrow{x, c^{x, 2^{-j}}_i}]],
\end{align*}
then if $\beta$ is the concave function in Definition~\ref{def: transportation cost} of admissibility,
\begin{align*}
   \sum_{j=k}^{\ell-1}\sum_{x\in L^{j}}\sum_{i=1}^{2^n}\tau(\mu[t](Q^{x, 2^{-j}}_i))\HOne{\overrightarrow{x, c^{x, 2^{-j}}_i}}
   &\leq\sqrt{n}\sum_{j=k}^{\ell-1} 2^{-j}\sum_{x\in L^{j}}\sum_{i=1}^{2^n}\beta(\mu[t](Q^{x, 2^{-j}}_i))\\
   &\leq \sqrt{n}\sum_{j=k}^{\ell-1} 2^{-j}2^{jn}\beta\left(\sum_{x\in L^{j}}\sum_{i=1}^{2^n}\frac{\mu[t](Q^{x, 2^{-j}}_i)}{2^{jn}}\right)\\
   &= \sqrt{n}\sum_{j=k}^{\ell-1} 2^{j(n-1)}\beta(2^{-jn}).
\end{align*} 
A similar calculation, recalling~\eqref{eqn: weak derivative weights}, yields
\begin{align*}
    \sum_{j=k}^{\ell-1}\sum_{x\in L^{j}}\sum_{i=1}^{2^n}\abs{(\mu[t](Q^{x, 2^{-j}}_i))'}\HOne{\overrightarrow{x, c^{x, 2^{-j}}_i}}
   &=\sqrt{n}\sum_{j=k}^{\ell-1} 2^{-j}\sum_{x\in L^{j}}\sum_{i=1}^{2^n}\abs{\nu[t](Q^{x, 2^{-j}}_i)}\\
   &\leq \sqrt{n}\sum_{j=k}^{\ell-1} 2^{-j}\TV{\nu[t]},
\end{align*}
where $\nu$ corresponds to $H$ from Remark~\ref{rem: weak derivative remark} with $\overstar{F}=\overstar{\mu}'$. Taking the $L^p$ norm, using~\eqref{eqn: weak measurable rep norm}, and recalling Lemma~\ref{lem: eliminating cycles}~\eqref{eqn: never cyclic energy} finishes the proof.
\end{proof}

An immediate corollary of the above is the following.
\begin{cor}\label{cor: exists finite energy}
    Under the same conditions as Proposition~\ref{prop: scaling paths}, for any $\mu^+$ and $\mu^-\in \overstar{W}^{1,p}(\prob(\compact))$, there exists $T\in \WStarPath(\mu^+, \mu^-)$ with $\ELp^{\tau, p}_{\lambda}(T)<\infty$.
\end{cor}
\begin{proof}
By a translation and dilation, without loss we may assume that $\compact \subset [-1,1)^n$. Note that $G^\ell_{\mu^\pm}$ as in Definition~\ref{def: n-adic mass fluxes} is $G^{0, \ell}_{\mu^\pm}$ in the notation of Proposition~\ref{prop: scaling paths}. Now define $\tilde{G}^\ell_{\mu^+}$ to be the discrete path whose underlying directed graph is the one for $G^\ell_{\mu^+}$ with its edges reversed, with $2^{-\ell}e_1$ added to all of its vertices. Also define 
\begin{align*}
G_\ell :&= \tilde{G}^\ell_{\mu^+} + G^\ell_{\mu^-}+[[\overrightarrow{2^{-\ell}e_1}, 0]],\\
a^-_\ell:&=P^\ell(\mu^-),\qquad
\Xi_\ell(x):=x+2^{-\ell}e_1,\qquad
a^+_\ell:=(\Xi_\ell)_\sharp P^\ell(\mu^+).
\end{align*}
Note that $\spt a^+_\ell\cap \spt a^-_\ell=\emptyset$, $\spt a^+_\ell\cup \spt a^-_\ell\subset [-2, 2)^n$, $G_\ell$ is never cyclic, and $G_\ell\in \WDPath(a^+_\ell, a^-_\ell)$. By Lemma~\ref{lem: Pk approximates}, we have $a^\pm_\ell\Wpflatto \mu^\pm$.
Since $G_\ell$ is never cyclic, using \cite[Lemma 2.9]{BrancoliniWirth18} with Lemma~\ref{lem: subadditive bound}, then Lemma~\ref{lem: eliminating cycles}~\eqref{eqn: never cyclic energy}, and Proposition~\ref{prop: scaling paths}, for some $C_\tau>0$ we have
\begin{align}
   &\norm{G_\ell}_{L^p(\vecmeas)}+\norm{G_\ell'}_{L^p(\vecmeas)}
   \leq C_\tau\DELp^{\tau, p}_\lambda(G_\ell)\notag\\
   &\leq 
   C_\tau\left(\tau(1)2^{-\ell}+2\sqrt{n}\left(\sum_{j=1}^\infty 2^{j(n-1)}\beta(2^{-jn})+\norm{(\overstar{\mu}^+)'}_{L^{p'}(\vecmeas)^\ast}+\norm{(\overstar{\mu}^-)'}_{L^{p'}(\vecmeas)^\ast}\right)\right)\label{eqn: energy bounded}
\end{align}
which is finite from \cite[Lemma 2.15]{BrancoliniWirth18}, hence $\norm{G_\ell}_{L^p(\vecmeas)}+\norm{G_\ell'}_{L^p(\vecmeas)}$ is bounded independent of $\ell\in \N$. 
Since $e\subset [-2, 2)^n$ for any $e\in E(G_\ell)$ and $\ell\in \N$ by construction, we can apply Proposition~\ref{prop: compactness result} to obtain (possibly passing to a subsequence) that $(G_\ell, a^+_\ell, a^-_\ell)\approxto (T, \mu^+, \mu^-)$ for some $T\in \WStarPath(\mu^+, \mu^-)$, and~\eqref{eqn: energy bounded} also shows that $\ELp^{\alpha, p}_\lambda(T)<\infty$.
\end{proof}
Combining with our compactness result, we can show existence of minimizers.
\begin{proof}[Proof of Theorem~\ref{thm: minimizers exist}]
    By the assumptions we may apply Corollary~\ref{cor: exists finite energy} to see that the set\\ $\WStarPath(\mu^+, \mu^-)$ is nonempty, and $\dist^{\tau, p}_\lambda(\mu^+, \mu^-)<\infty$. Then combining Lemma~\ref{lem: mass bounded}~\eqref{eqn: bounded seq for distance} and \eqref{eqn: same compact support} with Proposition~\ref{prop: compactness result} yields a minimizing $T\in \WStarPath(\mu^+, \mu^-)$ that is H\"older continuous from $[0, 1]$ to $(\vecmeas, \TV{\cdot})$, with approximating sequence $(G_i, a^+_i, a^-_i)\approxto (T, \mu^+, \mu^-)$, proving Theorem~\ref{thm: minimizers exist}~\eqref{eqn: minimizer exists}.

    To show Theorem~\ref{thm: minimizers exist}~\eqref{eqn: 1 rectifiable}, suppose \eqref{eqn: tau at zero condition} also holds. By Fatou's lemma we find that
    \begin{align*}
        \int_0^1 \left(\liminf_{i\to\infty}\Mcur^\tau(G_i[t])\right)dt
        &\leq \liminf_{i\to\infty}\int_0^1 \Mcur^\tau(G_i[t])dt\leq \liminf_{i\to\infty}\DELp^{\tau, p}_\lambda(G_i)<\infty.
    \end{align*}
    In particular $\liminf_{i\to\infty}\Mcur^\tau(G_i[t])$ is finite for a.e. $t\in [0, 1]$, thus for any $t$ in a full measure subset of $[0, 1]$ we can pass to a subsequence (not relabeled, but depending on $t$) such that
    \begin{align*}
        G_i[t]\flatto T[t],\qquad
        \sup_{i\in \N}\Mcur^\tau(G_i[t])<\infty.
    \end{align*}
    Then we can apply \cite[Lemma 5.1]{ColomboDeRosaMarcheseStuvard17} to find for all such $t\in [0, 1]$, the current $T[t]$ is rectifiable. By \cite[Step 3, Theorem 2.5]{TehYang21} and the $\TV{\cdot}$-H\"older continuity of $T$ in $t$, this shows $T[t]$ is rectifiable for \emph{all} $t\in[0,1]$; let $T[t]=[[E_t, V_t, \theta_t]]$. Let $D$ be a countable, dense subset of $[0, 1]$ 
    and define
    \begin{align*}
        E:=\bigcup_{t\in [0, 1]\cap D}E_t,
    \end{align*}
    which we see is countably $1$-rectifiable. Now for any $t_0\in [0, 1]\setminus D$, take $\{t_\ell\}_{\ell\in \N}\subset [0, 1]\cap D$ converging to $t_0$, then since $\theta_{t_\ell}\equiv 0$ on $\R^n\setminus E$, we find that
    \begin{align*}
        \abs{\int_{E_{t_0}\setminus E}\theta_{t_0}(x)d\mathcal{H}^1(x)}
        &\leq\int_{E_{t_0}\setminus E}\abs{\theta_{t_0}(x)-\theta_{t_\ell}(x)}d\mathcal{H}^1(x)\\
        &\leq \TV{T[t_\ell]-T[t_0]}\leq C_p \abs{t_\ell-t_0}^{1-1/p}\to 0,\qquad \ell \to\infty,
    \end{align*}
    hence $\theta_{t_0}$ is zero $\mathcal{H}^1$-a.e. on $E_{t_0}\setminus E$, finishing the proof of our theorem.
\end{proof}
\section{Metric properties}\label{sec: metric}
In this section, we prove Theorem~\ref{thm: metric}, i.e. that $\dist^{\tau, p}_\lambda$ is a metric. To this end, we show some auxiliary results. First, by the following argument similar to \cite[Theorem 4.6]{Santambrogio15book} we can relate $\Fcur$ to $\dist^{\tau, p}_\lambda$. 
\begin{lem}\label{lem: Lid and branched}
    If $1<p\leq \infty$, $\lambda>0$, and $\mu^\pm\in \overstar{W}^{1, p}(\prob(K))$ for some compact, convex $K\subset \domain$, then 
    \begin{align*}
        \dist^{\tau, p}_\lambda(\mu^+, \mu^-)
        &\geq \rho(\tau, 1)^{-1}\Lp{\Fcur(\mu^+-\mu^-)}+\lambda \Lp{\Fcur(\nu^+-\nu^-)},
    \end{align*}
    where $\nu^\pm$ are the $H$ obtained from Remark~\ref{rem: weak derivative remark} with $\overstar{F}=(\overstar{\mu}^\pm)'$.
\end{lem}
\begin{proof}
    Take $\{a^\pm_i\}_{i\in \N}\subset W^{1, p}(\atomic)$ and $G_i\in \WDPath(a^+_i, a^-_i)$ obtained from Lemma~\ref{lem: mass bounded}~\eqref{eqn: same compact support}. Then for any $\psi\in C^\infty_c(\domain)$ and $t\in [0, 1]$, if $\max(\norm{\psi}_{C(\domain)}, \norm{d\psi}_{C(\domain)})\leq 1$ we see that
    \begin{align*}
        \TV{G_i[t]}
        &\geq -\inner{G_i[t]}{d\psi}
        =\int_{\domain} \psi d(a^+_i[t]-a^-_i[t]),\\
        \TV{G'_i[t]}
        &\geq -\inner{G'_i[t]}{d\psi}
        =\int_{\domain} \psi d((a^+_i)'[t]-(a^-_i)'[t]).
    \end{align*}
    Taking a supremum over such $\psi$,
    \begin{align*}
        \TV{G_i[t]}
        &\geq \Fcur(a^+_i[t]-a^-_i[t]),\qquad
        \TV{G'_i[t]}
        \geq \Fcur((a^+_i)'[t]-(a^-_i)'[t]).
    \end{align*}
    Note that since $\mathcal{D}(\domain)$ is separable, the expressions on the right of each inequality above are countable suprema of Borel functions, hence also Borel. Integrating, and since $G_i$ is never cyclic we can use \cite[Lemma 2.9]{BrancoliniWirth18} with Lemma~\ref{lem: subadditive bound} and Lemma~\ref{lem: eliminating cycles}~\eqref{eqn: never cyclic energy}, to obtain
    \begin{align*}
        \DELp^{\tau, p}_\lambda(G_i)
        &\geq \rho(\tau, 1)\Lp{\TV{G_i}}+\lambda \Lp{\TV{G'_i}}\\
        &\geq \rho(\tau, 1)\Lp{\Fcur(a^+_i-a^-_i)}+\lambda \Lp{\Fcur((a^+_i)'-(a^-_i)')}.
    \end{align*}
    Since $a^\pm_i\Wpflatto \mu^\pm$, taking a limit infimum above finishes the proof.
\end{proof}
To show the triangle inequality, we will take spatial mollifications of elements of $\overstar{W}^{1, p}(\prob)$.
\begin{lem}\label{lem: mollifications}
    Let $\zeta\in C^\infty_c(B_1(0))$ be a nonnegative radial function 
    whose total integral is $1$, 
    and for any $\varepsilon>0$, let $\zeta_{\varepsilon}(x):=\varepsilon^{-n}\zeta(x/\varepsilon)$. 
        Let $\mu\in \overstar{W}^{1, p}(\prob)$ and define for any $t\in [0, 1]$ and Borel set $Q\subset \domain$,
    \begin{align*}
        (\zeta_{\varepsilon}\ast \mu)[t](Q)
        :&=\int_{Q}\int_{\R^n}\zeta_{\varepsilon}(z-y)d\mu[t](y)d\mathcal{L}^n(z).
    \end{align*}
    Then for any Borel $Q\subset \domain$ with $\mathcal{L}^n(Q)<\infty$, the above function belongs to $W^{1,p}$ with weak derivative given by
    \begin{align}\label{eqn: weak derivative of mollification}
        t\mapsto \int_{Q}\int_{\domain}\zeta_{\varepsilon}(z-y)d\nu[t](y)d\mathcal{L}^n(z),
    \end{align}
    where $\nu$ is the map corresponding to $H$ obtained from Remark~\ref{rem: weak derivative remark} with $\overstar{F}=\overstar{\mu}'$. In particular, for any $k\in \N$ we have $P_k( \zeta_{\varepsilon}\ast \mu)\in W^{1, p}(\atomic)$.

\end{lem}
\begin{proof}
    First, since $\zeta$ is bounded with compact support, the function $y\mapsto \int_Q\zeta_\varepsilon(z-y)d\mathcal{L}^n(z)$ belongs to $C_0(\domain)$, hence both $(\zeta_{\varepsilon}\ast \mu)[\cdot](Q)$ and the expression in~\eqref{eqn: weak derivative of mollification} are Borel. The function $(\zeta_{\varepsilon}\ast \mu)[\cdot](Q)$ is also clearly bounded and compactly supported, so belongs to $L^p$. By~\eqref{eqn: weak measurable rep norm}, we also have
    \begin{align*}
        \norm{\int_{Q}\int_{\domain}\zeta_{\varepsilon}(z-y)d\nu[\cdot](y)d\mathcal{L}^n(z)}_{L^p}
        &\leq \mathcal{L}^n(Q)\norm{\zeta_\varepsilon}_{C(\domain)}\norm{\TV{\nu[\cdot]}}_{L^p}\\
        &=\mathcal{L}^n(Q)\norm{\zeta_\varepsilon}_{C(\domain)}\norm{\overstar{\mu}'}_{L^p(\meas)}<\infty.
    \end{align*}
    
    Next for fixed $z\in \domain$, we write $(R_z)_\sharp \zeta_{\varepsilon}(y)=\zeta_{\varepsilon}(z-y)$. Then for $\eta\in C^\infty_c((0, 1))$, we may apply Fubini's theorem to find
    \begin{align*}
        \int_0^1 \eta'(t)(\zeta_{\varepsilon}\ast \mu)[t](Q)dt
        &=\int_{Q}\left(\int_0^1\eta'(t)\int_{\R^n}\zeta_{\varepsilon}(z-y)d\mu[t](y)dt\right)d\mathcal{L}^n(z)\\
        &=-\int_{Q}\inner{\overstar{\mu}'}{\eta(R_z)_\sharp \zeta_{\varepsilon}}d\mathcal{L}^n(z)\\
        &=-\int_0^1\eta(t)\int_{Q}\inner{\nu[t]}{(R_z)_\sharp \zeta_{\varepsilon}}d\mathcal{L}^n(z)dt,
    \end{align*}
    thus the weak derivative of $(\zeta_{\varepsilon}\ast \mu)[\cdot](Q)$ is given by~\eqref{eqn: weak derivative of mollification}, hence $(\zeta_{\varepsilon}\ast \mu)[\cdot](Q)\in W^{1,p}$.
    
    Finally, by definition, $P_k( \zeta_{\varepsilon}\ast \mu)$ is a sum of delta measures with coefficients given by $(\zeta_{\varepsilon}\ast \mu)[\cdot](Q)$ for certain cubes $Q$. By Lemma~\ref{lem: Lp weights}, this finishes the proof.
\end{proof}
We can now prove $\dist^{\tau, p}_\lambda$ is a metric. The triangle inequality involves a very delicate gluing procedure, and accordingly is the most difficult, and (unfortunately unavoidably) notationally heavy part of the proof.
\begin{proof}[Proof of Theorem~\ref{thm: metric}]
    It is clear that $\dist^{\tau, p}_\lambda$ is symmetric and nonnegative. Now suppose $\dist^{\tau, p}_\lambda(\mu_1, \mu_2)=0$ for some $\mu_1$, $\mu_2\in \overstar{W}^{1, p}(\prob(K))$. By Lemma~\ref{lem: Lid and branched}, we have $\Fcur(\mu_1[t]-\mu_2[t])=0$, or equivalently $\mu_1[t]=\mu_2[t]$ for a.e. $t\in [0, 1]$, hence $\mu_1=\mu_2$.

    Next suppose $\mu_1$, $\mu_2$, $\mu_3\in \overstar{W}^{1, p}(\prob(K))$, and let $\{(G_{12, i}, a^+_{12, i}, a^-_{12, i})\}_{i\in \N}$ and \\$\{(G_{23, i}, a^+_{23, i}, a^-_{23, i})\}_{i\in \N}$ be sequences coming from Lemma~\ref{lem: mass bounded}~\eqref{eqn: bounded seq for distance} whose energies converge to $\dist^{\tau, p}_\lambda(\mu_1, \mu_2)$ and $\dist^{\tau, p}_\lambda(\mu_2, \mu_3)$ respectively; also let  $\beta$ be the concave function from the definition of admissibility of $\tau$. Again by translation and dilation, we may assume $K\subset [-1, 1)^n$, while for each fixed $i\in \N$, by making sufficiently small uniform translations of $V(G_{23, i})$ we may assume that $V(G_{12, i})\cap V(G_{23, i})=\emptyset$.
    
    Now fixing $k\in \N$ with $k>1$, and $\varepsilon\in (0, 2^{-k})$, we construct discrete $W^{1, p}$ paths similar to \cite[Proposition 2.22]{BrancoliniWirth18}. First define $G^1_{i, k}\in \WDPath(a^-_{12, i}, P^k(a^-_{12, i}))$ by
    \begin{align*}
        G^1_{i, k}
        :=\sum_{x\in L^{\ell_{i, k}-1}}\sum_{j=1}^{2^n}\sum_{y\in \spt(a^-_{12, i})\cap Q^{x, 2^{1-\ell_{i, k}}}_j}a^-_{12, i}[\cdot](\{y\})[[\overrightarrow{y, c^{x, 2^{1-\ell_{i, k}}}_j}]]+G^{k, \ell_{i, k}}_{a^-_{12, i}},
    \end{align*}
    where $\ell_{i, k}\in \N$, $\ell_{i, k}>k$ is to be determined and $G^{k, \ell_{i, k}}_{a^-_{12, i}}$ is as in Proposition~\ref{prop: scaling paths}. Let $\beta$ be the concave function in the definition of admissibility. 
    Using that $a^-_{12, i}[t]$ is a probability measure for each $t$ and $\tau$ is increasing, and Proposition~\ref{prop: scaling paths} we calculate 
    \begin{align}
        \MLp^\tau_p(G^1_{i, k})
        &\leq \Lp{\sum_{x\in L^{\ell_{i, k}-1}}\sum_{j=1}^{2^n}\sum_{y\in \spt(a^-_{12, i})\cap Q^{x, 2^{1-\ell_{i, k}}}_j}\tau(a^-_{12, i}[\cdot](\{y\}))}2^{1-\ell_{i, k}}\sqrt{n}
        +\MLp^\tau_p(G^{k, \ell_{i, k}}_{a^-_{12, i}})\notag\\
        &\leq 2^{1-\ell_{i, k}}\sqrt{n}\tau(1)(\sharp\spt a^-_{12, i})+\sqrt{n}\sum_{j=k}^\infty 2^{j(n-1)}\beta(2^{-jn}),\label{eqn: G1 M bound}
    \end{align}
    and for each $t\in [0, 1]$,
    \begin{align}
        \TV{(G^1_{i, k})'[t]}
        &\leq \sqrt{n}\TV{(a^-_{12, i})'[t]}\left(2^{1-\ell_{i, k}}+\sum_{j=k}^\infty 2^{-j}\right).\label{eqn: G1' bound}
    \end{align}
    
    Next, let $\Xi^-_i(x):=x+v^-_i$ for some $v^-_i\in \R^n$ with $\abs{v^-_i}<1$ to be determined, take $\zeta$ as in Lemma~\ref{lem: mollifications}, and $\varepsilon\in (0, 1)$ also to be determined, then we will define a discrete $W^{1, p}$ path $G^2_{i, k}\in \WDPath(P^k(a^-_{12, i}), (\Xi^-_i)_\sharp P^k(\zeta_{\varepsilon}\ast a^-_{12, i}))$ by the following construction: given any $x_1$, $x_2\in L^{k-1}$ and $1\leq j_1, j_2\leq 2^n$, $G^2_{i, k}$ will have an edge $\overrightarrow{ c^{x_1, 2^{1-k}}_{j_1}, \Xi^-_i(c^{x_2, 2^{1-k}}_{j_2})}$
    with weight function given for each $t\in [0, 1]$ by
    \begin{align*}
        w_{x_1, x_2, j_1, j_2}(t):=\left(\zeta_{\varepsilon}\ast (\mathds{1}_{Q^{x_1, 2^{1-k}}_{j_1}}a^-_{12, i})\right)[t](Q^{x_2, 2^{1-k}}_{j_2}).
    \end{align*}
    Then we can use Fubini's theorem to find
    \begin{align*}
        \sum_{j_2=1}^{2^n}\sum_{x_2\in L^{k-1}}w_{x_1, x_2, j_1, j_2}(t)
        &=\left(\zeta_{\varepsilon}\ast (\mathds{1}_{Q^{x_1, 2^{1-k}}_{j_1}}a^-_{12, i})\right)[t]([-2, 2)^n))\\
        &=\int_{Q^{x_1, 2^{1-k}}_{j_1}}\left(\int_{\R^n}\zeta_{\varepsilon}(z-y)d\mathcal{L}^n(z)\right)d(a^-_{12, i})[t](y)\\
        &=a^-_{12, i}[t](Q^{x_1, 2^{1-k}}_{j_1})
        =(P^k(a^-_{12, i}))[t](\{c^{x_1, 2^{1-k}}_{j_1}\}),\\
        \sum_{j_1=1}^{2^n}\sum_{x_1\in L^{k-1}}w_{x_1, x_2, j_1, j_2}(t)
        &=\left(\zeta_{\varepsilon}\ast (\sum_{j_1=1}^{2^n}\sum_{x_1\in L^{k-1}}\mathds{1}_{Q^{x_1, 2^{1-k}}_{j_1}}a^-_{12, i})\right)[t](Q^{x_2, 2^{1-k}}_{j_2})\\
        &=\left(\zeta_{\varepsilon}\ast a^-_{12, i}\right)[t](Q^{x_2, 2^{1-k}}_{j_2})
        =P^k\left(\zeta_{\varepsilon}\ast a^-_{12, i}\right)[t](\{c^{x_2, 2^{1-k}}_{j_2}\}).
    \end{align*}
        Also fix any $\eta\in C^\infty_c((0, 1))$ and suppose 
        \begin{align*}
            a^-_{12, i}[t]
            =\sum_{\ell=1}^{L^-_i}a^-_{12, i, \ell}(t)\delta_{x^-_{i, \ell}}.
        \end{align*}
        Since each $a^-_{12, i}[t]$ is a probability measure, applying Fubini's theorem and Lemma~\ref{lem: Lp weights},
    \begin{align*}
        &\int_0^1\eta'(t)w_{x_1, x_2, j_1, j_2}(t)dt
        =\int_{Q^{x_2, 2^{1-k}}_{j_2}}\sum_{x^-_{i, \ell}\in Q^{x_1, 2^{1-k}}_{j_1}}\zeta_{\varepsilon}(z-x^-_{i, \ell})\left(\int_0^1\eta'(t)a^-_{12, i, \ell}(t)dt\right)d\mathcal{L}^n(z)\\
        &=-\int_0^1\eta(t)\int_{Q^{x_2, 2^{1-k}}_{j_2}}\sum_{x^-_{i, \ell}\in Q^{x_1, 2^{1-k}}_{j_1}}\zeta_{\varepsilon}(z-x^-_{i, \ell})(a^-_{12, i, \ell})'(t)d\mathcal{L}^n(z)dt\\
        &=-\int_0^1\eta(t)\int_{Q^{x_1, 2^{1-k}}_{j_1}}\int_{Q^{x_2, 2^{1-k}}_{j_2}}\zeta_{\varepsilon}(z-y)d\mathcal{L}^n(z)d (a^-_{12, i})'[t](y)dt,\\
        &\Lp{\int_{Q^{x_1, 2^{1-k}}_{j_1}}\int_{Q^{x_2, 2^{1-k}}_{j_2}}\zeta_{\varepsilon}(z-y)d\mathcal{L}^n(z)d (a^-_{12, i})'[\cdot](y)}\notag\\
        &\leq \Lp{\TV{(a^-_{12, i})'[\cdot]}}\sup_{y\in Q^{x_1, 2^{1-k}}_{j_1}}\int_{Q^{x_2, 2^{1-k}}_{j_2}}\zeta_{\varepsilon}(z-y)d\mathcal{L}^n(z)
        \leq\norm{(a^-_{12, i})'}_{L^p(\meas)}<\infty.
    \end{align*}
    Since $0
  \leq w_{x_1, x_2, j_1, j_2}\leq 1$, we have  $G^2_{i, k}\in \WDPath(P^k(a^-_{12, i}), (\Xi^-_i)_\sharp P^k(\zeta_{\varepsilon}\ast a^-_{12, i}))$. 
    It is clear that $w_{x_1, x_2, j_1, j_2}(t)$ is  nonzero only if the $\varepsilon$ neighborhood of $Q^{x_1, 2^{1-k}}_{j_1}$ has nonempty intersection with $Q^{x_2, 2^{1-k}}_{j_2}$; since $\varepsilon<2^{-k}$ this only holds for adjacent hypercubes. 
    The above calculation also yields 
    \begin{align*}
       \sum_{x_1, x_2\in L^{k-1}}\sum_{j_1, j_2=1}^{2^n}w_{x_1, x_2, j_1, j_2}(t)
       &=\sum_{x_2\in L^{k-1}}\sum_{j_2=1}^{2^n}P_k\left(\zeta_{\varepsilon}\ast a^-_{12, i}\right)[t](\{c^{x_2, 2^{1-k}}_{j_2}\})=1.
    \end{align*}
    Since at most $3^n2^{kn}$ pairs of the $Q^{x_i, 2^{1-k}}_{j}$ are adjacent, using $\tau(0)=0$, Jensen's inequality, and concavity of $\beta$ we obtain,
    \begin{align}
        \MLp^\tau_p(G^2_{i, k})
        &\leq \Lp{\sum_{Q^{x_i, 2^{1-k}}_{j_1}\text{ adj. }Q^{x_2, 2^{1-k}}_{j_2}}\beta(w_{x_1, x_2, j_1, j_2})}(2^{2-k}\sqrt{n}+\abs{v^-_i}+\varepsilon)\notag\\
        &\leq(2^{2-k}\sqrt{n}+\abs{v^-_i}+\varepsilon)3^n2^{kn}\beta(3^{-n}2^{-{kn}})
        \leq 3^n(2^{2-k}\sqrt{n}+\abs{v^-_i}+\varepsilon)2^{kn}\beta(2^{-kn}).\label{eqn: G2 M bound}
    \end{align}
    Also we calculate,
\begin{align*}
    &\sum_{x_1, x_2\in L^{k-1}}\sum_{j_1, j_2=1}^{2^n}\abs{w'_{x_1, x_2, j_1, j_2}(t)}\\
    &=\sum_{x_1, x_2\in L^{k-1}}\sum_{j_1, j_2=1}^{2^n}\abs{\int_{Q^{x_1, 2^{1-k}}_{j_1}}\int_{Q^{x_2, 2^{1-k}}_{j_2}}\zeta_{\varepsilon}(z-y)d\mathcal{L}^n(z)d (a^-_{12, i})'[t](y)}\\
    &\leq\sum_{x_1, x_2\in L^{k-1}}\sum_{j_1, j_2=1}^{2^n}\int_{Q^{x_2, 2^{1-k}}_{j_2}}\sum_{x^-_{i, \ell}\in Q^{x_1, 2^{1-k}}_{j_1}}\zeta_{\varepsilon}(z-x^-_{i, \ell})\abs{(a^-_{12, i, \ell})'(t)}d\mathcal{L}^n(z)\\
    &=\sum_{x_1\in L^{k-1}}\sum_{j_1=1}^{2^n}\sum_{x^-_{i, \ell}\in Q^{x_1, 2^{1-k}}_{j_1}}\abs{(a^-_{12, i, \ell})'(t)}\int_{\domain}\zeta_{\varepsilon}(z-x^-_{i, \ell})d\mathcal{L}^n(z)\\
    &=\sum_{x_1\in L^{k-1}}\sum_{j_1=1}^{2^n}\sum_{x^-_{i, \ell}\in Q^{x_1, 2^{1-k}}_{j_1}}\abs{(a^-_{12, i, \ell})'(t)}
    \leq \TV{(a^-_{12, i})'[t]},
\end{align*}
     hence we have
    \begin{align}
        \norm{(G^2_{i, k})'}_{L^p(\vecmeas)}
        &\leq (2^{2-k}\sqrt{n}+\abs{v^-_i}+\varepsilon)\Lp{\sum_{x_1, x_2\in L^{k-1}}\sum_{j_1, j_2=1}^{2^n}\abs{w'_{x_1, x_2, j_1, j_2}}}\notag\\
        &\leq(2^{2-k}\sqrt{n}+\abs{v^-_i}+\varepsilon)\norm{(a^-_{12, i})'}_{L^p(\meas)}.\label{eqn: G2' M bound}
    \end{align}
    Third, let $\Xi^+_i(x):=x+v^+_i$ for yet another $v^+_i\in \R^n$ with $\abs{v^+_i}<1$ to be determined, and writing
    \begin{align*}
        P^-_{k, i, j}:=P_k(\zeta_{\varepsilon} \ast a^-_{12, i})[\cdot](\{c^{x, 2^{1-k}}_j\}),
        \qquad 
        P^+_{k, i, j}:=P_k(\zeta_{\varepsilon} \ast a^+_{23, i})[\cdot](\{c^{x, 2^{1-k}}_j\}),
    \end{align*}
    and $a_+:=\max(a, 0)$ for the positive part of $a\in \R$, we define $G^3_{i, k}\in W^{1, p}(\vecmeas)$ by
\begin{align*}
    &G^3_{i, k}
    =\sum_{x\in L^{k-1}}\sum_{j=1}^{2^n}\left(
    \left(P^-_{k, i, j}-P^+_{k, i, j}\right)_+[[\overrightarrow{\Xi_i^-(c^{x, 2^{1-k}}_j), 0}]]\right.\\
    &+\left(P^+_{k, i, j}-P^-_{k, i, j}\right)_+[[\overrightarrow{0, \Xi_i^+(c^{x, 2^{1-k}}_j})]]
    +\left.\min\left(P^-_{k, i, j}, P^+_{k, i, j}\right)[[\overrightarrow{\Xi_i^-(c^{x, 2^{1-k}}_j), \Xi_i^+(c^{x, 2^{1-k}}_j})]]\right).
\end{align*}
    Then
\begin{align}
    \MLp^\tau_p(G^3_{i, k})
    &\leq 4\sqrt{n}\Lp{\sum_{x\in L^{k-1}}\sum_{j=1}^{2^n}\tau\left(\abs{P^-_{k, i, j}-P^+_{k, i, j}}\right)}\notag\\
    &+\left(\abs{v^+_i}+\abs{v^-_i}\right)\Lp{\sum_{x\in L^{k-1}}\sum_{j=1}^{2^n}\tau\left(\abs{\min\left(P^-_{k, i, j}, P^+_{k, i, j}\right)}\right)}\notag\\
    &\leq 4\sqrt{n}\sum_{x\in L^{k-1}}\sum_{j=1}^{2^n}\Lp{\beta\left(\abs{P^-_{k, i, j}-P^+_{k, i, j}}\right)}
    +\left(\abs{v^+_i}+\abs{v^-_i}\right)2^{kn}\tau(1).\label{eqn: convolution M ineq 1}
\end{align}
    Now for a fixed $t\in [0, 1]$, $x\in L^{k-1}$, and $1\leq j\leq 2^n$, we have,
    \begin{align*}
        \int_{Q^{x, 2^{1-k}}_{j}}\zeta_{\varepsilon}(z-y)d\mathcal{L}^n(z)&\leq 1,\qquad
        \abs{d_y\int_{Q^{x, 2^{1-k}}_{j}}\zeta_{\varepsilon}(z-y)d\mathcal{L}^n(z)}&\leq \varepsilon^{-1}\int_{\domain} \abs{d\zeta}d\mathcal{L}^n=:C_\varepsilon.
    \end{align*}
Thus we may use Fubini's theorem to write
\begin{align*}
    &\abs{(\zeta_{\varepsilon} \ast  a^+_{23, i})[t](Q^{x, 2^{1-k}}_j)-(\zeta_{\varepsilon} \ast  a^-_{12, i})[t](Q^{x, 2^{1-k}}_j)}\\
    &=\abs{\int_{\domain}\int_{Q^{x, 2^{1-k}}_j}\zeta_{\varepsilon}(z-y)d\mathcal{L}^n(z)d( a^+_{23, i}[t]-a^-_{12, i}[t])(y)}
    \leq C_\varepsilon\Fcur(a^+_{23, i}[t]-a^-_{12, i}[t]),
\end{align*}
 passing to a subsequence, the last expression above converges to zero as $i\to\infty$ for a.e. $t\in [0, 1]$. Because $\beta$ is concave and finite valued, it is also continuous, hence we obtain that \\$\beta\left(C_\varepsilon\Fcur(a^+_{23, i}[t]-a^-_{12, i}[t])\right)\to 0$ for each $t\in [0, 1]$ as $i\to\infty$. At the same time,
\begin{align*}
    &\beta\left(C_\varepsilon\Fcur(a^+_{23, i}[t]-a^-_{12, i}[t])\right)
     \leq \beta(C_\varepsilon\TV{a^+_{23, i}[t]-a^-_{12, i}[t]})
        \leq \beta(2C_\varepsilon),
\end{align*}
since the first expression above is zero for $t$ outside of a compact interval, if $1<p<\infty$ and $\abs{v^\pm_i}\to 0$, we can apply dominated convergence to~\eqref{eqn: convolution M ineq 1} to find
\begin{align}\label{eqn: G3 M bound}
    \liminf_{i\to\infty}\MLp^\tau_p(G^3_{i, k})=0.
\end{align}
Since $\beta$ must be increasing, we have $\norm{\beta\left(C_\varepsilon\Fcur(a^+_{23, i}-a^-_{12, i})\right)}_{L^\infty}=\beta\left(C_\varepsilon\norm{\Fcur(a^+_{23, i}-a^-_{12, i})}_{L^\infty}\right)$, hence~\eqref{eqn: G3 M bound} also holds when $p=\infty$. 
We also find
    \begin{align}
        &\norm{(G^3_{i, k})'}_{L^p(\vecmeas)}\notag\\
        &\leq 2\sqrt{n}\sum_{x\in L^{k-1}}\sum_{j=1}^{2^n}\Lp{\left(P^-_{k, i, j}-P^+_{k, i, j}\right)'}\notag\\
        &+\left(\abs{v^+_i}+\abs{v^-_i}\right)\sum_{x\in L^{k-1}}\sum_{j=1}^{2^n}\left[\Lp{\left(P^-_{k, i, j}\right)'}+\Lp{\left(P^+_{k, i, j}\right)'}\right]\notag\\
        &\leq 2\sqrt{n}\sum_{x\in L^{k-1}}\sum_{j=1}^{2^n}\Lp{\int_{\domain}\left(\int_{Q^{x, 2^{1-k}}_{j}}\zeta_{\varepsilon}(z-y)d\mathcal{L}^n(z)\right)d \left((a^-_{12, i})'[\cdot]-(a^+_{23, i})'[\cdot]\right)(y)}\notag\\
        &+\left(\abs{v^+_i}+\abs{v^-_i}\right)\sum_{x\in L^{k-1}}\sum_{j=1}^{2^n}\left[\Lp{(a^-_{12, i})'}+\Lp{(a^+_{23, i})'}\right]\notag\\
        &\leq 2\sqrt{n}C_\varepsilon\sum_{x\in L^{k-1}}\sum_{j=1}^{2^n}\Lp{\Fcur((a^-_{12, i})'[\cdot]-(a^+_{23, i})'[\cdot])}\notag\\
        &+\left(\abs{v^+_i}+\abs{v^-_i}\right)2^{kn}\left[\Lp{(a^-_{12, i})'}+\Lp{(a^+_{23, i})'}\right]
        \to 0\label{eqn: G3' M bound}
    \end{align}
    as $i\to\infty$.

   We can construct $G^4_{i, k}\in \WDPath((\Xi^+_i)_\sharp P^k(\zeta_{\varepsilon}\ast a^+_{23, i}), P^k(a^+_{23, i}))$ as we constructed $G^2_{i, k}$, replacing $P^k(a^-_{12, i})$ and $(\Xi^-_i)_\sharp P^k(\zeta_{\varepsilon}\ast a^-_{12, i})$ by $P^k(a^+_{23, i})$ and $(\Xi^+_i)_\sharp P^k(\zeta_{\varepsilon}\ast a^+_{23, i})$, then reversing the direction of the edges. Similarly, let $G^5_{i, k}\in \WDPath(P^k(a^+_{23, i}), a^+_{23, i})$ be obtained from the same construction as $G^1_{i, k}$ replacing $a^-_{12, i}$ and $P^k(a^-_{12, i})$ by $a^+_{23, i}$ and $P^k(a^+_{23, i})$, then reversing all edges. Now we define 
    \begin{align*}
        G_{i, k}:=G_{12, i}+G^1_{i, k}+G^2_{i, k}+G^3_{i, k}+G^4_{i, k}+G^5_{i, k}+G_{23, i},
    \end{align*}
    we can see that  $G_{i, k}\in \WDPath(a^+_{12, i}, a^-_{23, i})$. We note that we may make choices of $v^\pm_i$ with arbitrarily small norm for which the sets $V(G_{12, i})$, $V(G_{23, i})$, $\{\Xi^\pm_i(c^{x, 2^{1-k}}_j)\}_{x\in L^{k-1}, j=1,\ldots, 2^n}$, \\and $\{c^{x, 2^{1-k}}_j\}_{x\in L^{k-1}, j=1,\ldots, 2^n}$ are all mutually disjoint; in particular
    $G_{i, k}$ is never cyclic, thus by Lemma~\ref{lem: eliminating cycles}~\eqref{eqn: never cyclic energy}, we have
    \begin{align*}
        \DELp^{\tau, p}_\lambda(G_{i, k})\leq \DELp^{\tau, p}_\lambda(G_{12, i})+\DELp^{\tau, p}_\lambda(G_{23, i})+\sum_{\ell=1}^5(\MLp^\tau_p(G^\ell_{i, k})+\lambda\norm{(G^\ell_{i, k})'}_{L^p(\vecmeas)}).
    \end{align*}
    Thus combining with \eqref{eqn: G1 M bound}, \eqref{eqn: G1' bound}, \eqref{eqn: G2 M bound}, \eqref{eqn: G2' M bound}, we find
    \begin{align*}
        \DELp^{\tau, p}_\lambda(G_{i, k})
        &\leq \DELp^{\tau, p}_\lambda(G_{12, i})
        +\DELp^{\tau, p}_\lambda(G_{23, i})\\
        &+2^{1-\ell_{i, k}}\sqrt{n}\tau(1)((\sharp\spt a^-_{12, i})+(\sharp\spt a^+_{23, i}))+2\sqrt{n}\sum_{j=k}^\infty 2^{j(n-1)}\beta(2^{-jn})\\
        &+\lambda\sqrt{n}(\norm{(a^-_{12, i})'}_{L^p(\meas)}+\norm{(a^+_{23, i})'}_{L^p(\meas)})\left(2^{1-\ell_{i, k}}+\sum_{j=k}^\infty 2^{-j}\right)\\
        &+3^n(2^{3-k}\sqrt{n}+\abs{v^-_i}+\abs{v^+_i}+2\varepsilon)2^{kn}\beta(2^{-kn})\\
        &+\lambda(2^{3-k}\sqrt{n}+\abs{v^-_i}+\abs{v^+_i}+2\varepsilon)(\norm{(a^-_{12, i})'}_{L^p(\meas)}+\norm{(a^+_{23, i})'}_{L^p(\meas)})\\
        &+\MLp^\tau_p(G^3_{i, k})+\lambda \norm{(G^3_{i, k})'}_{L^p(\vecmeas)}.
    \end{align*}
    For each $i$ and $k$, we take $\ell_{i, k}>k$ and also large enough that $2^{1-\ell_{i, k}}\sqrt{n}\tau(1)((\sharp\spt a^-_{12, i})+(\sharp\spt a^+_{23, i}))<2^{-i}$. 
    Since the above is bounded uniformly in $i\in \N$ and $G_{i, k}$ is never cyclic, we may apply  Proposition~\ref{prop: compactness result} to pass to a subsequence and assume $(G_{i, k}, a^+_{12, i}, a^-_{23, i})\approxto (T, \mu_1, \mu_3)$ as $i\to\infty$ for some $T\in \WStarPath(\mu_1, \mu_3)$. For each $i\in \N$, we may choose $\abs{v^\pm_i}$ small enough that $\abs{v^\pm_i}<2^{-i}$, then taking $i\to\infty$ and recalling~\eqref{eqn: G3 M bound}, and~\eqref{eqn: G3' M bound} yields
    \begin{align*}
        &\dist^{\tau, p}_\lambda(\mu_1, \mu_3)
        \leq \liminf_{i\to\infty}\DELp^{\tau, p}_\lambda(G_{i, k})
        \leq \dist^{\tau, p}_\lambda(\mu_1, \mu_2)
        +\dist^{\tau, p}_\lambda(\mu_2, \mu_3)\\
        &+2\sqrt{n}\sum_{j=k}^\infty 2^{j(n-1)}\beta(2^{-jn})
        +2\lambda\sqrt{n}\sup_{i\in \N}\left(\norm{(a^-_{12, i})'}_{L^p(\meas)}+\norm{(a^+_{23, i})'}_{L^p(\meas)}\right)\sum_{j=k}^\infty 2^{-j}\\
        &+3^n(2^{3-k}\sqrt{n}+2\varepsilon)2^{kn}\beta(2^{-kn})
        +\lambda(2^{3-k}\sqrt{n}+2\varepsilon)\sup_{i\in \N}\left(\norm{(a^-_{12, i})'}_{L^p(\meas)}+\norm{(a^+_{23, i})'}_{L^p(\meas)}\right),
    \end{align*}
    then taking $\varepsilon\to 0$, followed by $k\to\infty$ and recalling \cite[Lemma 2.15]{BrancoliniWirth18} implies all but the first two terms on the right hand side of the above inequality converge to zero, thus the triangle inequality holds.
\end{proof}

Finally, we can also say what kind of convergence the metric $\dist^{\tau, p}_\lambda$ metrizes.

\begin{thm}\label{thm: weak convergence}
    Let $\{\mu_j\}_{j\in \N}\cup \{\mu\}\subset \overstar{W}^{1, p}(\prob(\compact))$, where $\compact$ is compact and convex, with\\ $\sup_{j\in \N}\Lp{\TV{\nu_j}}<\infty$ where $\nu_j$, $\nu$ are from Remark~\ref{rem: weak derivative remark} corresponding to $\overstar{\mu}'_j$, $\overstar{\mu}'$. Then 
    \begin{align*}
        \lim_{j\to \infty} &\dist^{\tau, p}_\lambda(\mu_j, \mu)=0
        \iff \mu_j\Wpflatto \mu.
    \end{align*}
\end{thm}
\begin{proof}
     Suppose that $\mu_j\Wpflatto \mu$. Given any subsequence, we may pass to a further subsequence to assume that $\mu_j[t]\flatto\mu[t]$ for a.e. $t\in [0, 1]$. For a fixed $i\in \N$, we can define $a^+_{i,j}:=P^i(\mu_j)$ and $a^-_i:=P^i(\mu)$, then $a^+_{i,j}\Wpflatto\mu_j$ and $a^-_i\Wpflatto\mu$ as $i\to\infty$ by Lemma~\ref{lem: Pk approximates}. We also pass to subsequences so that $a^+_{i,j}[t]\flatto\mu_j[t]$ and $a^-_i[t]\flatto\mu[t]$ for a.e $t\in [0, 1]$. Then for any $k\in \N$ with $k\geq j$, we can construct the analogues of the paths $G^\ell_{i, k}$ for $\ell=1, \ldots, 5$ from the proof of Theorem~\ref{thm: metric}, with $a^+_{i, j}$ and $a^-_i$ in place of $a^-_{12, i}$ and $a^+_{23, i}$ there, let us call these paths $G^\ell_{i, j, k}$  and write $G_{i, j, k}:=\sum_{\ell=1}^5G^\ell_{i, j, k}$. Again by Proposition~\ref{prop: compactness result} we can pass to a subsequence and assume $(G_{i, j, k}, a^+_{i, j}, a^-_i)\approxto (T_{j, k}, \mu_j, \mu)$ as $i\to\infty$ for some $T_{j, k}\in \WStarPath(\mu_j, \mu)$. Then recalling \eqref{eqn: G1 M bound}, \eqref{eqn: G1' bound}, \eqref{eqn: G2 M bound}, and \eqref{eqn: G2' M bound}, after taking $i\to\infty$ we have
     \begin{align*}
         \dist^{\tau, p}_\lambda(\mu_j, \mu)
         &\leq 2\sqrt{n}\sum_{j'=k}^\infty 2^{j'(n-1)}\beta(2^{-j'n})
         +4^n(2^{3-k}\sqrt{n}+2\varepsilon)2^{2kn}\beta(2^{-2kn})\\
        &+\lambda(\sqrt{n}\sum_{j'=k}^\infty 2^{-j'}+2^{3-k}\sqrt{n}+2\varepsilon)\sup_{j'\in \N}(\norm{\overstar{\mu}'_{j'}}_{L^{p'}(C_0(\domain))^\ast}+\norm{\overstar{\mu}'}_{L^{p'}(C_0(\domain))^\ast})\\
        &+\liminf_{i\to\infty}(\MLp^\tau_p(G^3_{i, j, k})+\lambda \norm{(G^3_{i, j, k})'}_{L^p(\vecmeas)}).
     \end{align*}
     Again as in the proofs of~\eqref{eqn: G3 M bound} and \eqref{eqn: G3' M bound}, we have the estimates
     \begin{align*}
         \liminf_{i\to\infty}\MLp^\tau_p(G^3_{i, j, k})
         &\leq 4\sqrt{n}\sum_{x\in L^{k-1}}\sum_{j'=1}^{2^n}\lim_{i\to\infty}\Lp{\beta\left(C_\varepsilon\Fcur(a^+_{i, j}-a^-_i)\right)}\\
         &= 4\sqrt{n}\sum_{x\in L^{k-1}}\sum_{j'=1}^{2^n}\Lp{\beta\left(C_\varepsilon\Fcur(\mu_j-\mu)\right)},\qquad 1<p<\infty,\\
         \liminf_{i\to\infty}\MLp^\tau_\infty(G^3_{i, j, k})
         &\leq 4\sqrt{n}\sum_{x\in L^{k-1}}\sum_{j'=1}^{2^n}\lim_{i\to\infty}\beta\left(C_\varepsilon\norm{\Fcur(a^+_{i, j}-a^-_i)}_{L^\infty}\right)\\
         &=4\sqrt{n}\sum_{x\in L^{k-1}}\sum_{j'=1}^{2^n}\beta\left(C_\varepsilon\norm{\Fcur(\mu_j-\mu)}_{L^\infty}\right),\\
         \liminf_{i\to\infty}\norm{(G^3_{i, j, k})'}_{L^p(\vecmeas)}
         &\leq 2\sqrt{n}\sum_{x\in L^{k-1}}\sum_{j'=1}^{2^n}\lim_{i\to\infty}\Lp{\Fcur((a^+_{i, j})'-(a^-_i)')}\\
         &=2\sqrt{n}\sum_{x\in L^{k-1}}\sum_{j'=1}^{2^n}\Lp{\Fcur(\nu_j-\nu)}.
     \end{align*}
     Taking $j\to\infty$ (using dominated convergence when $1<p<\infty$), then $\varepsilon\to 0$, and finally $k\to\infty$ shows $\dist^{\tau, p}_\lambda(\mu_j, \mu)$ converges to $0$ as $j\to\infty$. Since this holds for a subsequence of any subsequence of the original $\mu_j$, the whole sequence converges to $\mu$ in $\dist^{\tau, p}_\lambda$. 
     
     The other direction is immediate from Lemma~\ref{lem: Lid and branched}, finishing the proof. 
\end{proof}
\section{Consistency with stationary problem}\label{sec: consistency}
In this final section, we show consistency of our time-dependent problem with original stationary problem, when the probability measure valued functions given as data are constant. First we recall some terminology related to the stationary problem, following \cite{BrancoliniWirth18}. If $\widehat{a}^\pm\in \prob^a$, then a \emph{discrete mass flux between $\widehat{a}^+$ and $\widehat{a}^-$} is a weighted graph $\widehat{G}$ with edge function $w: E(\widehat{G})\to (0, \infty)$ satisfying Definition~\ref{def: discrete transport path} conditions~\eqref{eqn: boundary in vertices} and \eqref{eqn: Kirchhoff} with weights independent of $t$. Then for $\widehat{\mu}^\pm\in \prob$, a vector measure $\widehat{T}\in \vecmeas$ is a \emph{mass flux between $\widehat{\mu}^+$ and $\widehat{\mu}^-$} if there exist $\{\widehat{a}^\pm_i\}_{i\in \N}\subset \prob^a$ and discrete mass fluxes $\widehat{G}_i$ between $\widehat{a}^+_i$ and $\widehat{a}^-_i$ such that $\widehat{G}_i$, $\widehat{a}^+_i$, $\widehat{a}^-_i$ weak-$\ast$ converge to $\widehat{T}$, $\widehat{\mu}^+$, $\widehat{\mu}^-$ respectively; we write $(\widehat{G}_i, \widehat{a}^+_i, \widehat{a}^-_i)\wstarto (\widehat{T}, \widehat{\mu}^+, \widehat{\mu}^-)$ to denote this. If $\tau$ is an admissible transportation cost such that \eqref{eqn: tau at zero condition} holds, by \cite[Proposition 2.8]{ColomboDeRosaMarcheseStuvard17} for any rectifiable $1$-current $\widehat{T}$ and $\widehat{\mu}^\pm\in \prob$, it holds that
\begin{align*}
    \Mcur^\tau(\widehat{T})=\inf\{\liminf_{i\to\infty}\Mcur^\tau(\widehat{G}_i)\mid (\widehat{G}_i, \widehat{a}^+_i, \widehat{a}^-_i)\wstarto (\widehat{T}, \widehat{\mu}^+, \widehat{\mu}^-)\}.
\end{align*}
Additionally, the \emph{cost distance} is defined by
\begin{align*}
    \dist^\tau(\widehat{\mu}^+, \widehat{\mu}^-):=\inf\{\Mcur^\tau(\widehat{T})\mid \widehat{T}\text{ is a mass flux between }\widehat{\mu}^+\text{ and } \widehat{\mu}^-\}.
\end{align*}
By \cite[Corollary 2.24]{BrancoliniWirth18}, if $\tau$ is admissible, $\dist^\tau$ is a metric on the set of probability measures with compact supports.
For the remainder of the paper, given  $\widehat{T}\in \vecmeas$, $\widehat{a}\in \prob^a$, or $\widehat{\mu}\in \prob$, we identify them with elements of $\overstar{W}^{1, p}(\vecmeas)$, $\overstar{W}^{1, p}(\prob^a)$, and $\overstar{W}^{1, p}(\prob)$ respectively, that are constant on $[0, 1]$. Clearly the weak-$\ast$ derivatives of such elements are identically the corresponding zero functional.
\begin{thm}\label{thm: stationary problem}
    Suppose $\widehat{\mu}^\pm\in \prob$ are supported on a common compact, convex set $\compact$, $\lambda\geq 0$, $1<p\leq \infty$, and $\tau$ is a transportation cost satisfying~\eqref{eqn: tau at zero condition}. Then $\dist^\tau(\widehat{\mu}^+, \widehat{\mu}^-)=\dist^{\tau, p}_\lambda(\widehat{\mu}^+, \widehat{\mu}^-)$. If $\tau$ is admissible, there exists some $\widetilde{T}\in \WStarPath(\widehat{\mu}^+, \widehat{\mu}^-)$ which is constant on $[0, 1]$ such that $\dist^{\tau, p}_\lambda(\widehat{\mu}^+, \widehat{\mu}^-)=\ELp^{\tau, p}_\lambda(\widetilde{T})$.
\end{thm}
We comment that whether \emph{all} minimizers of the time-dependent problem must be constant in time is a more subtle question, which we leave for a future work.
\begin{proof}
    If $\widehat{T}$ is a mass flux between $\widehat{\mu}^+$ and $\widehat{\mu}^-$ and $(\widehat{G}_i, \widehat{a}^+, \widehat{a}^-)\wstarto (\widehat{T}, \widehat{\mu}^+, \widehat{\mu}^-)$, we can see \\$(\widehat{G}_i, \widehat{a}^+, \widehat{a}^-)\approxto (\widehat{T}, \widehat{\mu}^+, \widehat{\mu}^-)$ in the sense of Definition~\ref{def: approximating seq}, thus 
    \begin{align*}
        \Mcur^\tau(\widehat{T})\geq \ELp^{\tau, p}_\lambda(\widehat{T})\geq \dist^{\tau, p}_\lambda(\widehat{\mu}^+, \widehat{\mu}^-),
    \end{align*}
and taking an infimum over mass fluxes yields $\dist^\tau(\widehat{\mu}^+, \widehat{\mu}^-)\geq \dist^{\tau, p}_\lambda(\widehat{\mu}^+, \widehat{\mu}^-)$.

Now suppose by contradiction that $\dist^\tau(\widehat{\mu}^+, \widehat{\mu}^-)> \dist^{\tau, p}_\lambda(\widehat{\mu}^+, \widehat{\mu}^-)+\varepsilon$ for some $\varepsilon>0$. Then there exists an approximating sequence $(G_i, a^+_i, a^-_i)\approxto (\widehat{T}, \widehat{\mu}^+, \widehat{\mu}^-)$ with $\DELp^{\tau, p}_\lambda(G_i)\leq \dist^\tau(\widehat{\mu}^+, \widehat{\mu}^-)-\varepsilon$, and passing to a subsequence, we may assume that $\Fcur(a^\pm[t]-\widehat{\mu})$ converges to $0$ for a.e. $t\in [0, 1]$. If $p<\infty$, by Chebyshev's inequality, if $i$ is sufficiently large we obtain 
\begin{align*}
   &\mathcal{L}^1(\{t\in [0, 1]\mid \Mcur^\tau(G_i[t])\geq \dist^{\tau}(\widehat{\mu}^+, \widehat{\mu}^-)-\frac{\varepsilon}{2}\})
  \leq (\dist^{\tau}(\widehat{\mu}^+, \widehat{\mu}^-)-\frac{\varepsilon}{2})^{-1} \norm{\Mcur^\tau(G_i)}_{L^p}\\
   &\leq (\dist^{\tau}(\widehat{\mu}^+, \widehat{\mu}^-)-\frac{\varepsilon}{2})^{-1}\DELp^{\tau, p}_\lambda(G_i)
   \leq \frac{\dist^{\tau, p}_\lambda(\widehat{\mu}^+, \widehat{\mu}^-)+\frac{\varepsilon}{4}}
{\dist^{\tau}(\widehat{\mu}^+, \widehat{\mu}^-)-\frac{\varepsilon}{2}}
\leq \frac{\dist^{\tau}(\widehat{\mu}^+, \widehat{\mu}^-)-\frac{3\varepsilon}{4}}
{\dist^{\tau}(\widehat{\mu}^+, \widehat{\mu}^-)-\frac{\varepsilon}{2}}<1.
\end{align*}
Taking $\delta:=1-\frac{\dist^{\tau}(\widehat{\mu}^+, \widehat{\mu}^-)-\frac{3\varepsilon}{4}}
{\dist^{\tau}(\widehat{\mu}^+, \widehat{\mu}^-)-\frac{\varepsilon}{2}}$, next we may apply Egorov's theorem to obtain a measurable $S\subset [0,1]$ with $\mathcal{L}^1(S)<\delta$ where 
$\Fcur(a^\pm[\cdot]-\widehat{\mu})$ converges uniformly to $0$ on $[0,1]\setminus S$. The above calculation shows that for all sufficiently large $i$, there exists $t_i\in [0,1]\setminus S$ such that $\Mcur^\tau(G_i[t_i])<\dist^{\tau}(\widehat{\mu}^+, \widehat{\mu}^-)-\frac{\varepsilon}{2}$, in particular $a^\pm_i[t_i]\wto \widehat{\mu}^\pm$. If $p=\infty$, we can obtain such $t_i$ by definition of the $L^\infty$ norm. Since each $G_i[t_i]$ is a discrete mass flux between $a^+_i[t_i]$ and $a^-_i[t_i]$, 
\begin{align*}
    \dist^{\tau}(\widehat{\mu}^+, \widehat{\mu}^-)
    &\leq \liminf_{i\to \infty}\Mcur^\tau(G_i[t_i])
    \leq\dist^{\tau}(\widehat{\mu}^+, \widehat{\mu}^-)-\frac{\varepsilon}{2},
\end{align*}
a contradiction. 
If $\tau$ is admissible, by \cite[Corollary 2.20]{BrancoliniWirth18} there exists a mass flux $\widehat{T}$ between $\widehat{\mu}^\pm$ achieving the minimum value in $\dist^\tau(\widehat{\mu}^+, \widehat{\mu}^-)$. Clearly $\widehat{T}\in \WStarPath(\widehat{\mu}^+, \widehat{\mu}^-)$ and 
\begin{align*}
    \dist^{\tau, p}_\lambda(\widehat{\mu}^+, \widehat{\mu}^-)
    &=\dist^\tau(\widehat{\mu}^+, \widehat{\mu}^-)
    =\Mcur^\tau(\widehat{T})
    =\ELp^{\tau, p}_\lambda(\widehat{T}).
\end{align*}
\end{proof}

\bibliographystyle{plain}
\bibliography{transportpaths}
\end{document}